\documentclass[12pt,centertags,oneside]{amsart}
\usepackage{amsmath,amstext,amsthm,amscd,typearea}
\usepackage{amssymb}
\usepackage{fancyhdr}
\usepackage[mathscr]{eucal}
\usepackage{mathrsfs}
\usepackage{concmath}
\usepackage{charter}  
\usepackage{dsfont}  
\usepackage{hyperref} 
\usepackage{color} 
\usepackage{titletoc}

\usepackage[a4paper,width=16.2cm,top=2.8cm,bottom=2.8cm]{geometry} 
 
\numberwithin{equation}{section}  

\newcommand{\field}[1]{\mathbb{#1}}
\newcommand{\Z}{\field{Z}}
\newcommand{\R}{\field{R}}
\newcommand{\C}{\field{C}} 
\newcommand{\N}{\field{N}}

 \def\cC{\mathscr{C}}

\def\cL{\mathscr{L}}

\def\mL{\mathcal{L}}

\newcommand{\boldsym}[1]{\boldsymbol{#1}}

\newcommand\bn{\boldsym{n}}

\def\Im{{\rm Im}}

\newcommand{\til}[1]{\widetilde{#1}}

\newcommand{\cali}[1]{\mathscr{#1}}
\newcommand{\cH}{\cali{H}}
 
\DeclareMathOperator{\End}{End}

\DeclareMathOperator{\Ker}{Ker}

\DeclareMathOperator{\Dom}{Dom}

\DeclareMathOperator{\rank}{rank}

\DeclareMathOperator{\supp}{supp}

\DeclareMathOperator{\vol}{vol}

\newcommand{\imat}{\sqrt{-1}}

\newcommand{\om}{\omega}
\newcommand{\ol}{\overline}
\newcommand{\ddbar}{\overline\partial}
\newcommand{\dbar}{\partial}
\newtheorem{thm}{Theorem}[section]
\newtheorem{lemma}[thm]{Lemma}
\newtheorem{prop}[thm]{Proposition}
\newtheorem{cor}[thm]{Corollary}
\theoremstyle{definition}
\newtheorem{rem}[thm]{Remark}
\theoremstyle{definition}
\newtheorem{defn}[thm]{Definition}

\newtheorem{exam}[thm]{Example}

\newcommand{\be}{\begin{eqnarray}}
\newcommand{\ee}{\end{eqnarray}}
\newcommand{\ov}{\overline}

\newcommand{\comment}[1]{}
\begin{document}       
\title 
{Cohomology dimension growth for Nakano q-semipositive line bundles}     
           
\author{Huan Wang}  
\address{Institute of Mathematics, Academia Sinica, Taiwan}
\email{huanwang@gate.sinica.edu.tw, huanwang2016@hotmail.com} 
\date{25. September. 2019}  
\maketitle        
  \begin{abstract}    
  	We study the cohomology with high tensor powers of Nakano $q$-semipositive line bundles on complex manifolds. We obtain the asymptotic estimates for the dimension of cohomology with high tensor powers of semipositive line bundles over $q$-convex manifolds and various possibly non-compact complex manifolds, in which the order of estimates are optimal. Besides, estimates for the modified Dirac operator on Nakano $q$-positive line bundle on almost complex manifolds are given.   
  \end{abstract} 
           
\section{Introduction} 
Let $X$ be a complex manifold of dimension $n$ and $(E,h^E)$ be a holomorphic Hermitian vector bundle over $X$. Let $\nabla^E$ be the holomorphic Hermitian connection of $(E,h^E)$ and $R^{(E,h^E)}=(\nabla
^E)^2$ be the curvature of $\nabla^E$. The Bochner-Kodaira-Nakano formula and its variation with boundary term, \cite{Dem:86,AV:65,Griffs:66,MM}, play the central role in various vanishing theorems on complex manifolds. The latter have important applications in complex differential and algebraic geometry, such as the characterization of projective manifolds \cite{Kod:54}, Moishezon manifolds \cite{Sil:84,Siu:85,Dem:85} and more recently the criterion for uniruledness and rationally connectedness and related results \cite{BBDP:13,CDP:15,Yang:17}. The key ingredient in these formulas is the curvature term
$[\sqrt{-1}R^{(E,h^E)}, \Lambda]$, 
where $\Lambda$ is the dual of Hermitian metric on manifolds. With appropriate assumptions on the positivity of $R^E$, one can achieve the curvature term is strictly-positive, i.e., the pointwise Hermitian product
$\left\langle[\sqrt{-1}R^{(E,h^E)}, \Lambda]s,s\right\rangle_h>0$ for forms $s$ with values in $E$,
which is enough to prove vanishing theorems in various situations, see \cite{Kob:87,Dem}.

Instead of the strict positivity, we consider the $q$-semipositivity, 
which was introduced in \cite{Siu:82} over K\"{a}hler manifolds. 
A holomorphic Hermitian line bundle on K\"{a}hler manifolds is called Nakano $q$-positive (resp.\ semipositive) means that at every point the sum of
any set of $q$ eigenvalues of the curvature form is positive 
(resp.\ nonnegative) when the eigenvalues are computed with respect to
the K\"{a}hler metric. Another definition of the  $q$-positivity is the Griffiths $q$-positive (resp.\
semipositive), which means that at every point the curvature form has at least $n-q+1$ positive (resp.\ semipositive) eigenvalues, see \cite[Chapter\,3, Section\,1, Definition 1.1]{Oh:82}, \cite{Siu:82} and \cite{M:92}. 
More precisely, a holomorphic Hermitian line  bundle $(L,h^L)$ over a Hermitian manifold $(X,\omega)$ is Nakano $q$-semipositive with respect to the Hermitian metric $\omega$ of $X$, if for any $(n,q)$-forms $s$,
$\left\langle[\sqrt{-1}R^{(L,h^L)},\Lambda]s,s\right\rangle_h \geq 0$, see Definition \ref{def_qsemip}, (\ref{hypo_semi}), (\ref{hypo_pos}) and \cite{Ohs:05}.
 In this setting, the vanishing of harmonic forms does not hold in general, 
 however, the dimension of harmonic spaces with values in high tensor power 
 of such line bundles still can be estimated, and moreover the estimate turns out 
 to be optimal, see \cite{BB:02}. The solution of Grauert-Riemanschneider conjecture 
 \cite{Sil:84,Siu:85,Dem:85} shows that if $R^{(L,h^L)}\geq 0$ 
 (i.e., Nakano $1$-semipositive) on a compact complex manifold $X$ then
 $\dim H^q(X,L^k)=o(k^n)$ as $k\rightarrow \infty$ for all $ q\geq 1$. 
Demaily's solution involves holomorphic Morse inequalities \cite{Dem:85}: $\dim H^q(X,L^k\otimes E)\leq \rank(E)\frac{k^n}{n!}\int_{X(q)}(-1)^q(\frac{\sqrt{-1}}{2\pi}R^{(L,h^L)})^n+o(k^n)$ as $k\rightarrow \infty$, where $E$ is an arbitray holomorphic vector bundle and $X(q)$ is the set where $\sqrt{-1}R^{(L,h^L)}$ has exactly $q$ negative eigenvalues and $n-q$ positive eigenvalues. 
We refer to \cite{MM} for a comprehensive account of Demaily's 
holomorphic Morse inequalites and Bergman kernel asymptotics. 
   
 Let now $E$ be an arbitrary holomorphic line bundle over $X$.
 Along the same lines, Berndtsson \cite{BB:02} showed that if $R^{(L,h^L)}\geq 0$ then $\dim H^q(X,L^k\otimes E)=O(k^{n-q})$ and it improves the estimate of Siu and Demailly, which gives only $\dim H^q(X,L^k\otimes E)=o(k^n)$ as $k\rightarrow\infty$ (since $X(q)$ is the empty set for a semipositive line bundle). The  magnitude $k^{n-q}$ is optimal. By adapting their methods to general (possibly non-compact) complex manifolds with $L^2$-cohomolgy \cite{Wh:16}, we obtain a local estimate of Bergman density function on compact subsets of the underling manifolds when $R^L\geq 0$. As applications, the estimates of the Berndtsson type still hold on covering manifolds, i.e., $\dim_\Gamma \ov H_{(2)}^{0,q}(X,L^k\otimes E)=O(k^{n-q})$ for all $ q\geq 1$, and
 $1$-convex manifolds, i.e., $\dim H^q(X,L^k\otimes E)=O(k^{n-q})$ for all $ q\geq 1$, see \cite{Wh:16, Wh:17}. With additional assumptions on the positivity of $(L,h^L)$, the same estimates hold on pseudoconvex domains, weakly $1$-complete manifolds and complete manifolds, see \cite{Wh:17}. Note that, on projective manifolds, the estimate of $O(k^{n-q})$ type for nef line bundles can be found in \cite{Dem:12}, and the case of pseudo-effective line bundles was obtained in \cite{Mats:14}. On an arbitrary compact manifolds, such estimates for semipositive line bundles equipped with Hermitian metric with analytic singularities were established by \cite{Wh:17,WZ:19} (in the latter paper a vector bundle $E$ of arbitrary rank is considered).
    
  In this paper, in order to generalize such estimates to $q$-convex manifolds,
   we use the notion of Nakano $q$-semipositivity from \cite{Siu:82,Ohs:05}, which includes the usual semipositivity $R^L\geq 0$ as a special case. We remark that, inspired by \cite{BB:02,MM}, this paper together with \cite{Wh:16,Wh:17} give a unified approach to the optimal estimate of the dimension of cohomology of high tensor powers of line bundles with semipositivity on (compact and non-compact) complex manifolds.
  
\begin{defn}[\cite{Siu:82}]\label{def_qsemip}  
	Let $X$ be a complex manifold of $\dim X=n$ and $\omega$ a Hermitian metric on $X$. Let $(L,h^L)$ be a holomorphic Hermitian line bundle over $X$. Let $1\leq q\leq n$.
	\begin{itemize}
		\item[(A)] $(L,h^L)$ is called \textbf{Nakano $q$-positive} (resp.\ semipositive, negative, seminegative) with respect to $\omega$ at $x\in X$, if the sum of
		any set of $q$ eigenvalues of the curvature form $R_x^{L}$ is positive (resp.\ nonnegative,
		negative, nonpositive) when the eigenvalues are computed with respect to
		the Hermitian metric $\omega$.
		\item[(B)] 
		$(L,h^L)$ is called \textbf{Griffiths $q$-positive} (resp.\ semipositive, negative, seminegative) at $x\in X$, if the curvature form $R_x^{L}$ has at least $n-q+1$ 
		positive (resp.\ semipositive, negative, seminegative) eigenvalues.
	\end{itemize}
\end{defn}

For the relation of the notions of Griffiths and Nakano $q$-positivity, see Remark \ref{D:q-pos-M}. The basic example of Nakano $q$-positivity is the dual of canonical bundle $K^*_X$ on a compact K\"{a}hler manifold $X$ of $\dim X=n$. With respect to a K\"{a}hler metric $\omega$,  the Ricci curvature of $X$ is positive (resp.\ nonnegative) if and only if $K_X^*$ is Nakano $1$-positive (resp.\ $1$-semipositive); the scalar curvature of $X$ is positive (resp.\ nonnegative) if and only if $K_X^*$ is Nakano $n$-positive (resp.\ $n$-semipositive).
The basic example of Griffiths $q$-positivity is the dual of tautological line bundle $L(E^*)^*$, which is Griffiths ($n+1$)-positive on the projective bundle $P(E^*)$ of a holomorphic Hermitian vector bundle $(E,h^E)$ over a compact complex manifold $X$ of $\dim X=n$.

Firstly, we provide a refined local estimate on Bergman density functions for Nakano $q$-semipositive line bundles, which generalizes the main result in \cite{BB:02,Wh:16} and \cite[Theorem 3.1]{Wh:17}. The advantage is that, it enables us to study the harmonic spaces of tensor powers of line bundles with weaker semipositivity on complex manifolds. 

\begin{thm}\label{thm_local}
	Let $(X,\omega)$ be a Hermitian manifold of dimension $n$ and let
	$(L,h^L)$ and $(E,h^E)$ be holomorphic Hermitian line bundles over $X$. Let $1\leq q\leq n$. Let $K\subset X$ be a compact subset and $(L,h^L)$ be Nakano $q$-semipositive with respect to $\omega$ on a neighborhood of $K$. Then there exists  $C>0$ depending on $K$, $\omega$, $(L,h^L)$ and $(E,h^E)$, such that
	\begin{equation}
	B^j_k(x) \leq Ck^{n-j}\quad \mbox{for all}~ x\in K, k\geq 1, q\leq j\leq n,
	\end{equation}
	where $B^j_k(x)$ is the Bergman density function \eqref{e:Bergfcn} of harmonic 
	$(0,j)$-forms with values in $L^k\otimes E$. In particular, if $(L,h^L)$ is semipositive on a neighborhood of $K$, the estimate holds on $K$ for all $k\geq 1$ and $1\leq j\leq n$.
\end{thm}  
  
As a direct application, it leads to the refinement of \cite[Theorem 1.1]{BB:02} and \cite[Theorem 1.2]{Wh:16} as follows, refer to Definition \ref{def_covering} for $\Gamma$-covering manifolds. 
\begin{cor}\label{thm_covering}
	Let $(X,\omega)$ be a $\Gamma$-covering manifold of dimension $n$, and
	let $(L,h^L)$ and $(E,h^E)$ be two $\Gamma$-invariant 
	holomorphic Hermitian line bundles on $X$. Let $1\leq q\leq n$ and $(L,h^L)$ be Nakano $q$-semipositive with respect to $\omega$ on $X$.
	Then there exists $C>0$ such that for any $k\geq 1$, $q\leq j\leq n$, we have
	\begin{equation}
	\dim_{\Gamma}{\overline{H}}^{0,j}_{(2)}(X, L^k\otimes E)=
	\dim_{\Gamma}\cH^{0,j}(X, L^k\otimes E)
	\leq C k^{n-j}.
	\end{equation}
	In particular, if $(L,h^L)$ is semipositive on $X$, 
	the estimate holds for all $k\geq 1$ and $1\leq j\leq n$.
\end{cor}   
Note that holomorphic Morse inequalities on covering manifolds
were obtained in \cite{TCM:01,MTC:02}. 

Secondly, we obtain a refined estimate of $L^2$-cohomology on  Hermitian manifolds from the local estimate of $B_k^j(x)$ as \cite[Theorem 1.1]{Wh:17}. It provides an uniform approach to study the cohomology of high tensor power of Nakano $q$-semipositive line bundles over various compact and non-compact manifolds. 

Let $(X,\omega)$ be a Hermitian manifold of dimension $n$. 
Let $dv_X:=\omega^n/n!$ be the volume form on $X$.
Let $(L,h^L)$ and $(E,h^E)$ be holomorphic Hermitian vector bundles on $X$ with $\rank(L)=1$.
We denote by $(L^2_{0,q}(X,L^k \otimes E),\|\cdot\|)$ the space
of square integrable $(0,q)$-forms with values in $L^k \otimes E$ 
with respect to the $L^2$ inner product induced by the above data.
We denote by $\ddbar^E_k$ the maximal extension of the Dolbeault operator
on $L^2_{0,\bullet}(X,L^k \otimes E)$ and by 
$\ddbar^{E*}_k$ its Hilbert space adjoint.
Let $\cH^{0,q}(X,L^k \otimes E)$ be the space 
of harmonic $(0,q)$-forms with values in $L^k \otimes E$ on $X$. For a given $0\leq q\leq n$, we say that 
\textbf{the concentration condition holds in bidegree $(0,q)$ 
	for harmonic forms with values in $L^k\otimes E$ for large $k$}, 
if there exists a compact subset $K\subset X$ and $C_0>0$ such 
that for sufficiently large $k$, we have
\begin{equation}
\|s\|^2\leq C_0\int_K |s|^2 dv_X,
\end{equation}
for $s\in \Ker(\ddbar^E_k)\cap \Ker(\ddbar^{E*}_k)\cap 
L^2_{0,q}(X,L^k\otimes E)$. The set $K$ is called the exceptional compact set 
of the concentration. 
We say that \textbf{the fundamental estimate holds in 
	bidegree $(0,q)$ for forms with values in $L^k\otimes E$ for large $k$}, 
if there exists a compact subset $K\subset X$ and $C_0>0$ 
such that for sufficiently large $k$, we have 
\begin{equation}
\|s\|^2\leq C_0(\|\ddbar^E_k s\|^2+\|\ddbar^{E,*}_k s\|^2+\int_K |s|^2 dv_X),
\end{equation}
for $s\in \Dom(\ddbar^E_k)\cap \Dom(\ddbar^{E*}_k)\cap L^2_{0,q}(X,L^k\otimes E)$.
The set $K$ is called the exceptional compact set of the estimate. 
 
\begin{thm}\label{thm_L2general}
	Let $(X,\omega)$ be a Hermitian manifold of dimension $n$ and let $(L,h^L)$ and $(E,h^E)$ be holomorphic Hermitian line bundles on $X$. Let $1\leq q\leq n$. Let the concentration condition
	holds in bidegree $(0,q)$ for harmonic forms with values in $L^k\otimes E$ for large $k$.
	Let $(L,h^L)$ be Nakano $q$-semipositive with respect to $\omega$ on a neighborhood
	of the exceptional set $K$.
	Then there exists $C>0$ such that for sufficiently large $k$, we have
	\begin{eqnarray}
	\dim \cH^{0,q}(X,L^k\otimes E)&\leq& Ck^{n-q}.
	\end{eqnarray}
	The same estimate also holds for reduced $L^2$-Dolbeault cohomology groups,
	\begin{equation}
	\dim \overline{H}^{0,q}_{(2)}(X,L^k\otimes E)\leq Ck^{n-q}.
	\end{equation}
	In particular, if the fundamental estimate holds in bidegree $(0,q)$ for forms with values in $L^k\otimes E$ for large $k$, the same estimate holds for $L^2$-Dolbeault cohomology groups,
	\begin{equation} 
	\dim H^{0,q}_{(2)}(X,L^k\otimes E)\leq Ck^{n-q}.
	\end{equation}
\end{thm}    
 
Finally, by Theorem \ref{thm_L2general}, we can study the dimension of cohomology on $q$-convex manifolds with semipositive line bundles. Holomorphic Morse inequalities for $q$-convex manifolds were obtained by Bouche \cite{Bouche:89} and \cite[Section 3.5]{MM}.
\begin{thm}\label{thm_main}
	Let $X$ be a $q$-convex manifold of dimension $n$ and $1\leq q\leq n$, and let $(L,h^L)$ and $(E,h^E)$ be holomorphic Hermitian line bundles on $X$.  Assume $(L,h^L)\geq 0$ on a neighborhood of the exceptional subset $K$.
	Then, there exists $C>0$ such that for every $j\geq q$ and $k\geq 1$,
	\begin{equation}
	\dim H^{j}(X,L^k\otimes E)\leq Ck^{n-j}. 
	\end{equation}  
\end{thm} 

The extremal case is also interesting when the $\omega$-trace of $R^{(L,h^L)}$ is non-negative (i.e., $n$-semipositive), see Sec.\ \ref{sec_prel}. 
We obtain the finiteness of dimension of cohomology of high tensor power of such line bundles in Sec.\ \ref{sec_pf_local} and \ref{sec_pf_main}.
Related to the Nakano $n$-semipositive and the $\omega$-trace of curvature tensor, a direct consequence from \cite{Yau:78}, \cite[Ch.III.(1.34)]{Kob:87} and \cite[5.1 Corollary]{CDP:15}, which strengthens \cite[Theorem B (A)]{Wu:81}, is as follows:
If a compact K\"{a}hler manifold $X$ possesses a quasi-positive $(1,1)$-form representing the first Chern class $c_1(X)$, then $X$ is projective and rationally connected. And a compact, simply connected, K\"{a}hler manifolds with nonnegative bisectional curvature is projective and rationally connected, see Proposition \ref{thm_rc} and \ref{cor_sc}. 
 
For the Nakano $q$-positive cases, inspired by \cite{M:92}, \cite[Theorem 1.1, 2.5]{MM:02} and \cite[Section 1.5]{MM}, we generalize the estimates of modified Dirac operator $D^{c,A}_k$ (see \cite[Definition 1.3.6, Section 1.5]{MM}) of high tensor powers of positive line bundles to the Nakano $q$-positive case for all $1\leq q\leq n$ as follows.  
\begin{thm}\label{thm_main_spectral}
	Let $(X,J)$ be a compact smooth manifold with almost complex structure $J$ and $\dim _{\R}X=2n$. Let $g^{TX}$ be a Riemannian metric compatible with $J$ and $\om:=g^{TX}(J\cdot,\cdot)$ be the real $(1,1)$-forms on $X$ induced by $g^{TX}$ and $J$. Let $(E,h^E)$ and $(L,h^L)$ be Hermitian vector bundles on $X$ with $\rank(L)=1$. Let $\nabla^E$ and $\nabla^L$ be Hermitian connections on $(E,h^E)$ and $(L,h^L)$ and let $R^E:=(\nabla^E)^2$ and $R^L:=(\nabla^L)^2$ be the curvatures. Let $\frac{\sqrt{-1}}{2\pi}R^L$ be compatible with $J$. Assume $1\leq q\leq n$ and $(L,h^L)$ is Nakano $q$-positive with respect to $\om$ on $X$ (see also (\ref{hypo_pos})). Then there exists $C_L>0$ such that for any $k\in \N$ and any $s\in \Omega^{0,\geq q}(X,L^k\otimes E)$,
	\be
	\|D_k^{c,A}s\|^2\geq (2\mu_qk-C_L)\|s\|^2,
	\ee
	where the constant $\mu_q>0$ defined in (\ref{eq_mu_def}).
	Especially, for $k$ large enough,
	\be
	\Ker\left(D^{c,A}_k|_{\Omega^{0,\geq q}(X,L^k\otimes E)}\right)=0.
	\ee
\end{thm}

This paper is organized as follows. In Sec.  \ref{sec_prel} we introduce the notions and basic facts on Definition \ref{def_qsemip}. In Sec. \ref{sec_pf_local} we provide the local estimate of the Bergam density function associated with Nakano $q$-semipositive line bundles, Theorem \ref{thm_local}, and its applications, Corollary \ref{thm_covering} and Theorem \ref{thm_L2general}.
In Sec. \ref{sec_pf_main} we prove Theorem \ref{thm_main} and related results. In Sec. \ref{sec_spec}, estimates for the modified Dirac operator on Nakano $q$-positive line bundle on almost complex manifolds, Theorem \ref{thm_main_spectral}, are given. From \cite{BB:02}, we see Theorem \ref{thm_local}, Corollary \ref{thm_covering}, Theorem \ref{thm_L2general} and \ref{thm_main} give the optimal order $O(k^{n-j})$ of dimension of the cohomology. And Theorem \ref{thm_main_spectral} provides a precise bound $\mu_q$ for $q$-positive line bundles along the lines of \cite{MM:02,MM}.  
For techniques and formulations of this paper, we refer the reader to \cite{BB:02,MM,Siu:85,Wh:16,Wh:17}.  

\section{Preliminaries}\label{sec_prel}
\subsection{$L^2$-cohomology}
Let $(X, \omega)$ be a Hermitian manifold of dimension $n$ and $(F, h^F)$ a holomorphic Hermitian vector bundle over $X$. Let $\Omega^{p,q}(X, F)$ be the space of smooth $(p,q)$-forms on $X$ with values in $F$ for $p,q\in \N$. The volume form is $dv_{X}:=\frac{\omega^n}{n!} $.

The $L^2$-scalar product is given by $\langle s_1,s_2 \rangle=\int_X \langle s_1(x), s_2(x) \rangle_h dv_X(x)$ on $\Omega^{p,q}(X, F)$, where $\langle\cdot,\cdot\rangle_h:=\langle\cdot,\cdot\rangle_{h^F,\omega}$ is the pointwise Hermitian inner product induced by $\omega$ and $h^F$. We denote by $L^2_{p,q}(X, F)$, the $L^2$ completion of $\Omega^{p,q}_0(X, F)$, which is the subspace of $\Omega^{p,q}(X, F)$ consisting of elements with compact support.
 
Let $\ddbar^{F}: \Omega_0^{p,q} (X, F)\rightarrow L^2_{p,q+1}(X,F) $ be the Dolbeault operator and let  $ \ddbar^{F}_{\max} $ be its maximal extension (see \cite[Lemma 3.1.2]{MM}). From now on we still denote the maximal extension by $ \ddbar^{F} :=\ddbar^{F}_{\max} $ and the associated Hilbert space adjoint by $\ddbar^{F*}:=\ddbar^{F*}_H:=(\ddbar^{F}_{\max})_H^*$. Consider the complex of closed, densely defined operators
$L^2_{p,q-1}(X,F)\xrightarrow{\ddbar^{F}}L^2_{p,q}(X,F)\xrightarrow{\ddbar^{F}} L^2_{p,q+1}(X,F)$. Note that
 $(\ddbar^{F})^2=0$. By \cite[Propositon 3.1.2]{MM}, the operator defined by
\begin{eqnarray}\label{eq40}\nonumber
\Dom(\square^{F})&=&\{s\in \Dom(\ddbar^{F})\cap \Dom(\ddbar^{F*}): 
\ddbar^{F}s\in \Dom(\ddbar^{F*}),~\ddbar^{F*}s\in \Dom(\ddbar^{F}) \}, \\ 
\square^{F}s&=&\ddbar^{F} \ddbar^{F*}s+\ddbar^{F*} \ddbar^{F}s \quad \mbox{for}~s\in \Dom(\square^{F}),
\end{eqnarray}
is a positive, self-adjoint extension of Kodaira Laplacian, called the Gaffney extension.

\begin{defn}[\cite{MM}]
	The space of harmonic forms $\cH^{p,q}(X,F)$ is defined by 
	\begin{equation}\label{eq45}
	\cH^{p,q}(X,F):=\Ker(\square^{F})=\{s\in \Dom(\square^{F})\cap L^2_{p,q}(X, F): \square^{F}s=0 \}.
	\end{equation}
	The $q$-th reduced $L^2$-Dolbeault cohomology is defined by 
	\begin{equation}\label{eq46}
	\overline{H}^{0,q}_{(2)}(X,F):=\dfrac{\Ker(\ddbar^{F})\cap  L^2_{0,q}(X,F) }{[ \Im( \ddbar^{F}) \cap L^2_{0,q}(X,F)]},
	\end{equation}
	where $[V]$ denotes the closure of the space $V$. The $q$-th (non-reduced) $L^2$-Dolbeault cohomology is defined by 
	\begin{equation}
	H^{0,q}_{(2)}(X,F):=\dfrac{\Ker(\ddbar^{F})\cap  L^2_{0,q}(X,F) }{ \Im( \ddbar^{F}) \cap L^2_{0,q}(X,F)}.
	\end{equation}
\end{defn}  

According to the general regularity theorem of elliptic operators, 
$s\in \cH^{p,q}(X,F) $ implies $s\in\Omega^{p,q}(X,F)$. By weak Hodge decomposition (cf.\ \cite[(3.1.21) (3.1.22)]{MM}),  
\begin{equation}\label{eq47}
\overline{H}^{0,q}_{(2)}(X,F)\cong \cH^{0,q}(X,F)
\end{equation} for any $q\in \N$. The $q$-th cohomology of the sheaf of holomorphic sections of $F$ is isomorphic to the the $q$-th Dolbeault cohomology, $H^q(X,F)\cong H^{0,q}(X,F)$. 

For a given $0\leq q \leq n$, we say \textbf{the fundamental estimate holds in bidegree $(0,q)$ for forms with values in $F$}, if there exists a compact subset $K\subset X$ and $C>0$ such that 
\begin{equation}
\|s\|^2\leq C(\|\ddbar^F s\|^2+\|\ddbar^{F*}\|^2+\int_K|s|^2dv_X),
\end{equation}
for $s\in \Dom(\ddbar^F)\cap\Dom(\ddbar^{F,*})\cap L^2_{0,q}(X,F)$. $K$ is called the exceptional compact set of the estimate. If the fundamental estimate holds, the reduced and non-reduced $L^2$-Dolbeault cohomology coincide, see \cite[Theorem 3.1.8]{MM}. For a given $0\leq q\leq n$, we say that \textbf{the concentration condition holds in bidegree $(0,q)$ for harmonic forms with values in $F$}, if there exists a compact subset $K\subset X$ and $C>0$ such that 
\begin{equation}
\|s\|^2\leq C\int_K |s|^2 dv_X,
\end{equation}
for $s\in \Ker(\ddbar^F)\cap \Ker(\ddbar^{F*})\cap L^2_{0,q}(X,F)$.
We call $K$ the exceptional compact set of the concentration. Note if the fundamental estimate holds, the concentration condition also.

\subsubsection{\textbf{The convexity of complex manifolds and $\Gamma$-coverings}}
\begin{defn}
	A complex manifold $X$ of dimension $n$ is called $q$-convex if there exists a smooth function $\varrho\in \cC^\infty(X,\R)$ such that the sublevel set $X_c=\{ \varrho<c\}\Subset X$ for all $c\in \R$ and the complex Hessian $\dbar\ddbar\varrho$ has $n-q+1$ positive eigenvalues outside a compact subset $K\subset X$. Here $X_c\Subset X$ means that the closure $\overline{X}_c$ is compact in $X$. We call $\varrho$ an exhaustion function and $K$ exceptional set. $X$ is $q$-complete if $K=\emptyset$ in additional.	
\end{defn}
 
Every compact complex manifold is $q$-convex for all $1\leq q\leq n$. By definition, a compact complex manifold is exactly a $0$-convex manifold. For non-compact manifolds, Greene-Wu \cite[Ch.\,IX.\,(3.5)\,Theorem]{Dem} showed that: 
	Every connected non-compact complex manifold of dimension $n$ is $n$-complete. Moreover, every connected complex manifold of dimension $n$ is $n$-convex. 
Thus, if $X$ is a connected non-compact complex manifold of dimension $n$ and $E$ a holomorphic vector bundle over $X$,
$H^n(X,E)=0$, see \cite{AG:62}.
We denote the $j$-th Dolbeault cohomology with compact supports by $[H^{0,j}(X,E)]_0$, see \cite[(20.8) (20.17)]{HL:88}. Note that if $X$ is compact, $[H^{0,j}(X,E)]_0$ is equal to the usual cohomology. The duality between it and the usual Dolbeault cohomology on $q$-convex manifold of dimension $n$ with $1\leq q\leq n$, is given by
\be\label{dual_convx}
\dim [H^{0,j}(X,E)]_0= \dim H^{0,n-j}(X,E^*\otimes K_X)\leq \infty\quad \mbox{for all}~ 0\leq j\leq n-q.
\ee
If $q=1$, then, moreover, $\dim [H^{0,n}(X,E)]_0=\dim H^0(X,E^*\otimes K_X)$, where $K_X=\wedge^n T^{1,0*}X$. 

Let $M$ be a relatively compact domain with smooth boundary $bM$ in a complex manifold $X$. Let $\rho\in \cC^\infty(X,\R)$ such that $M=\{ x\in X: \rho(x)<0 \}$ and $d\rho\neq 0$ on $bM=\{x\in X: \rho(x)=0\}$. We denote the closure of $M$ by $\overline{M}=M\cup bM$. We say that $\rho$ is a defining function of $M$. Let $T^{(1,0)}bM:=\{ v\in T_x^{(1,0)}X: \dbar\varrho(v)=0 \}$ be the analytic tangent bundle to $bM$ at $x\in bM$. The Levi form of $\rho$ is the $2$-form $\cL_\rho:=\dbar\ddbar\rho\in \cC^\infty(bM, T^{(1,0)*}bM\otimes T^{(0,1)*}bM)$. $M$ is called strongly (resp.\ (weakly)) pseudoconvex if the Levi form $\cL_\rho$ is positive definite (resp.\ semidefinite). Note any strongly pseudoconvex domain is $1$-convex.

	A complex manifold $X$ is called weakly $1$-complete if there exists a smooth plurisubharmonic function $\varphi\in \cC^\infty(X,\R)$ such that $\{x\in X: \varphi(x)<c\}\Subset X$ for any $c\in \R$. $\varphi$ is called an exhaustion function. Note any $1$-convex manifold is weakly $1$-complete.
 
	A Hermitian manifold $(X,\omega)$ is called complete, if all geodesics are defined for all time for the underlying Riemannian manifold. 

\begin{defn}\label{def_covering}
	Let $(X,\omega)$ be a Hermitian manifold of dimension $n$ on which a discrete 
	group $\Gamma$ acts holomorphically, freely and properly such that $\omega$ is a $\Gamma$-invariant Hermitian 
	metric and the quotient $X/\Gamma$ is compact. We say $X$ is a $\Gamma$-covering manifold, see also \cite{Atiyah:76, MM, Wh:16}.
\end{defn}	
 
\subsubsection{\textbf{Kodaira Laplacian with $\ddbar$-Neumann boundary conditions}}

Let $(X,\omega)$ be a Hermitian manifold of dimension $n$ and $(F,h^F)$ be a holomorphic Hermitian vector bundles over $X$. Let $M$ be a relatively compact domain in $X$. Let $\rho$ be a defining function of $M$ satisfying $M=\{ x\in X: \rho(x)<0 \}$ and $|d\rho|=1$ on $bM$, where the pointwise norm $|\cdot|$ is given by $g^{TX}$ associated to $\omega$.

Let $e_{\bn} \in TX$ be the inward pointing unit normal at $bM$ and $e_{\bn}^{(0,1)}$ its projection on $T^{(0,1)}X$. In a local orthonormal frame $\{ w_1,\cdots,\omega_n \}$ of $T^{(1,0)}X$, we have $e_{\bn}^{(0,1)}=-\sum_{j=1}^n w_j(\rho)\ov w_j$. Let $B^{0,q}(X,F):=\{ s\in \Omega^{0,q}(\ov M, F): i_{e_{\bn}^{(0,1)}}  s=0 ~\mbox{on}~bM \}$. Then $B^{0,q}(M,F)=\Dom(\ddbar_H^{F*})\cap \Omega^{0,q}(\overline{M},F)$ and the Hilbert space adjoint $\ddbar_H^{F*}$ of $\ddbar^F$ coincides with the formal adjoint $\ddbar^{F*}$ of $\ddbar^F$ on $B^{0,q}(M,F)$, see \cite[Proposition 1.4.19]{MM}. The operator $\square_N s:=\ddbar^{F}\ddbar^{F*}s+\ddbar^{F*}\ddbar^{F}s$ for $s\in \Dom(\square_N):=\{s\in B^{0,q}(M,F): \ddbar^Fs\in B^{0,q+1}(M,F) \}$. The Friedrichs extension of $\square_N$ is a self-adjoint operator and is called the Kodaira Laplacian with $\ddbar$-Neumann boundary conditions, which coincides with the Gaffney extension of the Kodaira Laplacian, see \cite[Proposition 3.5.2]{MM}. $\Omega^{0,\bullet}(\overline{M},F)$ is dense in $\Dom(\ddbar^F)$ in the graph-norms of $\ddbar^F$, and $B^{0,\bullet}(M,F)$ is dense in $\Dom(\ddbar^{F*}_H)$ and in $\Dom(\ddbar^F)\cap \Dom(\ddbar^{F*}_H)$ in the graph-norms of $\ddbar^{F*}_H$ and $\ddbar^E+\ddbar^{E*}_H$, respectively, see \cite[Lemma 3.5.1]{MM}. Here the graph-norm is defined by $\|s\|+\|Rs\|$ for $s\in \Dom(R)$. 

 \subsection{Nakano $q$-semipositive line bundles and the $\omega$-trace}
  Let $(X,\omega)$ be a Hermitian manifold of dimension $n$ and $(E,h^E)$ a holomorphic Hermitian vector bundle over $X$. Let $\nabla^E$ be the holomorphic Hermitian connection of $(E,h^E)$ and $R^{(E,h^E)}=(\nabla
  ^E)^2$ be the curvature. Let $\bigwedge^{p,q}T^*X:=\bigwedge^p T^{1,0*}X\otimes \bigwedge^q T^{0,1*} X$ and let $\bigwedge^{p,q}T_x^*X$ be the fibre of the bundle $\bigwedge^{p,q}T^*X$ for $x\in X$, and $\Omega^{p,q}(X,E):=\cC^\infty(X,\bigwedge^{p,q}T^*X\otimes E)$ the space of smooth $(p,q)$-forms with values in $E$. We set $\langle,\rangle_h$ the induced pointwise Hermitian metric in the context.
  
  Let $(L,h^L)$ be a holomorphic Hermitian line bundle over $X$. Then $R^L=\ddbar\dbar \log|s|^2_{h^{L}}$ for any local holomorphic frame $s$, and the Chern-Weil form of the first Chern class of $L$ is $c_1(L, h^L)=\frac{\sqrt{-1}}{2\pi}R^L$, which is a real $(1,1)$-form on $X$. We use the notion of positive $(p,p)$-form, see \cite[Chapter III, \S 1, (1.1) (1.2) (1.5) (1.7)]{Dem}. If a $(p,p)$-form $T$ is positive, we write $T\geq 0$.  Let $\Lambda$ be the dual of the operator $\mL:=\omega\wedge\cdot$ on $\Omega^{p,q}(X)$ with respect to the Hermitian inner product $\langle,\rangle_h$ on $X$. In a local orthonormal frame $\{ w_j \}_{j=1}^n$ of $T^{1,0}X$ and its dual $\{ w^j \}$ of $T^{1,0*}X$, $R^{(L,h^L)}=R^{(L,h^L)}(w_i,\ov w_j)w^i\wedge \ov w^j$, $\mL=\sqrt{-1}\sum_{j=1}^n w^j\wedge \ov w^j\wedge \cdot$ and $\Lambda=-\sqrt{-1}\sum_{j=1}^n i_{\ov w^j} i_{w^j}$. For $s\in \Omega^{p,q}(X)$,  $\left\langle[\sqrt{-1}R^{(L,h^L)}, \Lambda]s,s\right\rangle_h\in \cC^\infty(X,\R)$.
 
 Recall the notion of $q$-semipositivity of line bundle in Definition \ref{def_qsemip}, see \cite[(1)]{Ohs:05} and \cite[Section 4]{Siu:82}. By the definition, for $1\leq q\leq n$, $(L,h^L)$ is \textbf{Nakano $q$-semipositive with respect to $\omega$ at $x\in X$}, means that
 \be\label{hypo_semi}
 \left\langle[\sqrt{-1}R^{(L,h^L)},\Lambda]\alpha,\alpha\right\rangle_h \geq 0\quad \mbox{for all}~ \alpha\in \wedge^{n,q}T_x^*X.
 \ee    
 We denoted it by $(\star_q)\geq 0$ at $x$. For $1\leq q\leq n$,
 $(L,h^L)$ is \textbf{Nakano $q$-positive with respect to $\omega$ at $x$}, means that 
 \be\label{hypo_pos}
 \left\langle[\sqrt{-1}R^{(L,h^L)},\Lambda]\alpha,\alpha\right\rangle_h > 0\quad \mbox{for all}~ \alpha\in \wedge^{n,q}T_x^*X\setminus\{0\}.
 \ee   
 We denoted it by $(\star_q)>0$ at $x$. For a subset $Y\subset Y$, $(L,h^L)$ is \textbf{Nakano $q$-semipositive (resp.\ positive) with respect to $\omega$ on $Y$}, if $(\star_q)\geq 0~(\mbox{resp.\ }>0)$ at every point of $Y$.
In a local orthonormal frame $\{\omega_j\}$ of $T^{1,0}X$ around $x$, (\ref{hypo_semi}) is equivalent to
\be \label{D:q-semi-local} 
\left\langle R^{(L,h^L)}(\omega_i,\ov\omega_j)\ov\omega^j\wedge i_{\ov \omega_i}\alpha,\alpha\right\rangle_h\geq 0\quad \mbox{for all}~ \alpha\in \wedge^{0,q}T_x^*X.
\ee
And (\ref{hypo_pos}) can be represented by replacing  $\bigwedge^{n,q}T_x^*X$ and $\geq$ by $\bigwedge^{n,q}T_x^*X\setminus\{0\}$ and $>$ in (\ref{D:q-semi-local}) respectively.

\begin{rem}[notions of the $q$-positivity]\label{D:q-pos-M}
	Note that the notion of Nakano $q$-positive depends on the 
	choice of Hermitian metric $\omega$. On the other hand, 
	the notion of Griffiths $q$-positive is independent of the choice of 
	$\omega$. If $L$ is Nakano $q$-positive at $x$, then $L$ is also Griffiths 
	$q$-positive at $x$. If $L$ is Griffiths $q$-positive at $x$ 
	then there exists a metric $\omega$ such that $L$ is Nakano $q$-positive at $x$ 
	with respect to $\omega$. Actually, if $L$ is Griffiths $q$-positive on $X$ 
	then for any compact set $K$ there exists a Hermitian metric $\omega$ on $X$
	such that $L$ is Nakano $q$-positive on $K$ with respect to $\omega$, 
	see \cite[(3.5.7)]{MM} or \cite[(9)]{M:92} for the construction of $\omega$. 
	The Nakano $1$-semipositivity (resp.\ positivity) coincides with the Griffiths $1$-semipositivity 
	(resp.\ positivity), i.e., the usual semipositivity (resp.\ positivity). By definition, 
	$q$-positivity implies the q-semipositivity.
	 For vector bundles, we refer to \cite{Oh:82,Siu:82} for the definitions of the $q$-positive 
	 in the sense of Nakano and Griffiths. 
\end{rem}

\subsubsection{\textbf{The special case $(\star_1)\geq 0$.}} 

An important special case is $(\star_1)\geq 0$, which is equivalent to $(L,h^L)$ is semipositive as follows.
 \begin{defn}
 	A holomorphic Hermitian line bundle $(L,h^L)$ is semipositive at $x\in X$, if $R^{(L,h^L)}(U,\ov U)\geq 0$ for $U\in T^{1,0}_xX$, denoted by $(L,h^L)_x\geq 0$. For a subset $Y\subset X$, $(L,h^L)$ is semipositive on $Y$ if $(L,h^L)_x\geq 0$ at all $x\in Y$. The definition of $(L,h^L)_x>0$ is analogue. 	
 \end{defn}    
 
 From the definition, $(L,h^L)_x\geq 0$ implies $(\star_q)\geq 0$ at $x$ for all $1\leq q\leq n$. Conversely, $(\star_1)\geq 0$ at $x$ implies $(L,h^L)_x\geq 0$. Thus, $(\star_q)\geq 0$ is a refinement of $(L,h^L)\geq 0$.
 
 \begin{prop}\label{P:semipos}
 	Let $(L,h^L)\geq 0$ (resp.\ $>0$) at $x\in X$. Then, for any Hermitian metric $\omega$ on $X$, $1\leq q \leq n$ and $\alpha\in \bigwedge^{n,q}T_x^*X$ $\left(\mbox{resp.\ } \bigwedge^{n,q}T_x^*X\setminus \{0\}\right)$,
 	\be
 	\left\langle[\sqrt{-1}R^{(L,h^L)},\Lambda]\alpha,\alpha\right\rangle_h \geq 0 ~(\mbox{resp.\ } >0).
 	\ee 
 \end{prop}

\begin{proof}
	Let $\{\omega_j \}$ be an orthonormal frame around $x$ such that $\sqrt{-1}R^{(L,h^L)}_x=\sqrt{-1}c_j(x)\omega^j\wedge\ov \omega^j$. Let $C_J(x):=\sum_{j\in J}c_j(x)$ for each ordered $J=(j_1,\cdots,j_q)$ with $|J|= q$. Let $\alpha\in \bigwedge^{n,q}T_x^*X$ and $\alpha=\sum _Jf_{NJ}(x) w^{N}\wedge\ov w^J$, $N=(1,\cdots,n)$ and $|J|=q$, 
	\begin{eqnarray}\label{eq_comp_starq}
	\left\langle[\sqrt{-1}R^{(L,h^L)},\Lambda]\alpha,\alpha\right\rangle_h(x)
	=\sum_J C_J(x) |f_{NJ}(x)|^2.
	\end{eqnarray}
	Since $(L,h^L)_x\geq 0$, $c_j(x)\geq 0$ for all $1\leq j\leq n$ and thus $C_J(x)\geq 0$ for all $1\leq |J|\leq n$.	And the positive case follows similarly.
		\end{proof} 
 
 \begin{prop}\label{P:1semipos}
 	$(L,h^L)_x\geq 0$ if and only if $(\star_1)\geq 0$ at $x$.
 \end{prop} 
 \begin{proof}
 	Suppose $(\star_1)\geq 0$ at $x$.
 	Let $U\in T^{1,0}_xX$ with $U=\sum_{k=1}^n u_k \omega_k$ in a local orthonormal frame $\{\omega_j \}_{j=1}^n$ of $T^{1,0}X$ around $x$. We set $\alpha=u_k\ov \omega^k\in T^{0,1*}_xX$, and then
 	$
 	R^{(L,h^L)}(U,\ov U)=\ov u_j R^L(\omega_i,\ov \omega_j)u_i=\langle R^{L}(\omega_i,\ov\omega_j)\ov\omega^j\wedge i_{\ov \omega_i}\alpha,\alpha\rangle_h\geq 0.
 	$
 \end{proof}
 
The general relation among $(\star_q)\geq 0$, $1\leq q\leq n$, is as follows.   
\begin{prop}\label{prop_refine_semip}
	If $(\star_q)\geq 0$ at $x$, then $(\star_{q+1})\geq 0$ at $x$.
\end{prop}  
\begin{proof}
	From $(\star_q)\geq 0$ at $x$ and (\ref{eq_comp_starq}), $C_J(x)\geq 0$ for each ordered $|J|=q$. Let $C_K(x):=\sum_{k\in K}c_k(x)$ for each ordered $|K|=q+1$. Then  $C_K(x)=\frac{1}{q}\sum_{|J|=q, J\subset K}C_J(x)\geq 0$, and thus $(\star_{q+1})\geq 0$ at $x$ by (\ref{eq_comp_starq}).
\end{proof}

\begin{rem}\label{rem_j_pos}
	Clearly, the positive case $(>)$ of Proposition \ref{P:1semipos} and \ref{prop_refine_semip} also hold.
\end{rem}

\subsubsection{\textbf{The special case $(\star_n)\geq 0$.}}
 
Another interesting case is $(\star_n)\geq 0$, which is equivalent to the $\omega$-trace of Chern curvature tensor $R^{(L,h^L)}$ is non-negative as follows.  
 \begin{defn}
 	The $\omega$-trace of Chern curvature tensor $R^{(L,h^L)}$, 
$\tau(L,h^L,\omega)\in \cC^\infty(X,\R)$, is defined by
$\sqrt{-1}R^{(L,h^L)}\wedge\omega_{n-1}=\tau(L,h^L,\omega)\omega_n$.
\end{defn}
 Equivalently, let $\{ w_j \}_{j=1}^n$ be a local orthonormal frame of $T^{(1,0)}X$ with respect to $\omega$ and $\{w^j \}$ the dual frame of $T^{(1,0)*}X$,  \be\nonumber
 	\tau(L,h^L, \omega):=
 	\mbox{Tr}_\omega R^{(L,h^L)}=\sum_{j=1}^n R^{(L,h^L)}(w_j,\ov{w}_j)=\sum_{i,k} R^{(L,h^L)}\left(\frac{\dbar}{\dbar z_i},\frac{\dbar}{\dbar \ov z_k}\right)\langle dz_i,d\ov z_k\rangle_{g^{T^*X}}.
 	\ee 
We say the $\omega$-trace of Chern curvature tensor $R^{(L,h^L)}$ is semipositive (resp.\ positve) if $\tau(L,h^L,\omega)\geq 0~(\mbox{resp.\ }>0)$. 
From (\ref{eq_comp_starq}), it follows immediately: 
 \begin{prop}\label{prop_tr_pos} We have the following
 		\begin{itemize}
 			\item[(1)]
 			$\tau(L,h^L,\omega)|\alpha|^2_h=\left\langle[\sqrt{-1}R^{(L,h^L)},\Lambda]\alpha,\alpha\right\rangle_h$\quad for all $\alpha\in \wedge^{n,n}T_x^*X$, $ x\in X$.
 			\item[(2)] 
 			$\tau(L,h^L,\omega)_x\geq 0$ if and only if $(\star_n)\geq 0$ at $x$.
 			\item[(3)]
 			$\tau(L,h^L,\omega)_x> 0$ if and only if $(\star_n)> 0$ at $x$.
 			\item[(4)] $\tau(L,h^L,\omega)=-\tau(L^*,h^{L^*},\omega)$.
 		\end{itemize}
 \end{prop} 
 
 \begin{exam}
 	Let $(X,\omega)$ be a K\"{a}hler manifold of dimension $n$, let $(K^*_X,h_\omega)$ be the dual of canonical line bundle $K_X:=\wedge^n T^{(1,0)*}X$ associated with the Hermitian metric $h_\omega$ induced from $\omega$.
 	The $\omega$-trace of $R^{(K_X^*,h_\omega)}$ coincides with the scalar curvature $r^X_\omega$ of $(X,\omega)$ up to the mutiplication of 2, i.e., for $r^X_\omega:=2\sum_jRic(\omega_j,\ov\omega_j)$,
 	\be
 	2\tau(K^*_X,h_\omega,\omega)=2\mbox{Tr}_\omega R^{(K_X^*,h_\omega)}=r_\omega^X.
 	\ee  
 \end{exam}
 
\subsubsection{\textbf{The $\omega$-trace of Chern curvature tensor of vector bundles.}} 
 Let $(E,h^E)$ be a holomorphic Hermitian vector bundle over a complex manifold $(X,\omega)$. 
The $\omega$-trace of Chern curvature tensor $R^{(E,h^E)}$, 
$\tau(E,h^E,\omega):=\mbox{Tr}_\omega R^{(E,h^E)}\in\cC^\infty(X,\End(E))$ is defined by
 	\be 
 	\sqrt{-1}R^{(E,h^E)}\wedge\omega_{n-1}=\tau(E,h^E,\omega)\omega_n,
 	\ee
 	see \cite[Section 1.5.]{CDP:15}. Note in \cite[(4.15)]{MM} $\Lambda_\omega(R^E)$ is the contraction of $R^E$ with respect to $\omega$ and thus $\sqrt{-1}\Lambda_\omega(R^E)=\tau(E,h^E,\omega)$ in our notations. We define  
 	$\tau(E,h^E,\omega)\geq 0$ (resp.\ $>0$) at $x\in X$ by
 	\be
 	\langle\tau(E,h^E,\omega)s,s\rangle_{h^E}\geq 0
 	\ee
 	(resp.\ $>0$) for $s\in E_x$ (resp.\ $s\in E_x\setminus \{0\}$). Similarly, we can define $\tau(E,h^E,\omega)\leq 0$ (resp.\ $<0$).  
  Let $(E^*,h^{E^*})$ be the dual bundle with the induced metric given by $(E,h^E)$, 
 \be
 \quad \tau(E,h^E,\omega)=-\tau(E^*,h^{E^*},\omega)
 \ee
 coincide as Hermitian matrices, see \cite{Kob:87}.
For the projection $\pi:P(E^*)\rightarrow X$ and the dual of tautological line bundle $O_{E^*}(1):=(L(E^*))^*$ over $P(E^*)$ with  Hermitian metrics $\omega_{P(E^*)}$ and $h^{O_{E^*}(1)}$ induced from $\omega$ and $h^E$, see \cite[Ch.\,III, Sec.\,5]{Kob:87}, we set 
\begin{equation}
\tau(O_{E^*}(1)):=\tau\left(O_{E^*}(1),h^{O_{E^*}(1)},\omega_{P(E^*)}\right).
\end{equation}

    \section{Bergman density function and applications}\label{sec_pf_local}
    \subsection{Local estimates for Bergman density functions}
    Let $(X,\omega)$ be a Hermitian (paracompact) manifold of dimension $n$ and
    $(L,h^L)$ and $(E,h^E)$ be Hermitian holomorphic line bundles over $X$.
    For $k\in\N$ we form the Hermitian line bundles $L^k:=L^{\otimes k}$ and
    $L^{k}\otimes E$, the latter
    endowed with the metric $h_k=(h^L)^{\otimes k}\otimes h^E$. Let $\nabla^L$ be the holomoprhic Hermitian connection of $(L,h^L)$.  The curvature of $(L, h^L)$ is defined by $R^L=(\nabla^L)^2$, then the Chern-Weil form of the first Chern class of $L$ is $c_1(L, h^L)=\frac{\sqrt{-1}}{2\pi}R^L$, which is a real $(1,1)$-form on $X$. 
     
    Let $dv_X:=\frac{\omega^n}{n!}$ be the volume form on $X$. We denote that the maximal extension of the Dolbeault operator $\ddbar^E_k:=(\ddbar^{L^k\otimes E})_{\max}$, its Hilbert space adjoint $\ddbar^{E*}_k:=(\ddbar^{L^k\otimes E})^*_H$, and the Gaffney extension of Kodaira Laplacian $\square^E_k:=\square^{L^k\otimes E}$. Let $\cH^{o,q}(X,L^k \otimes E):=\Ker(\square^E_k)\cap L^2_{0,q}(X,L^k \otimes E)$ be the space of harmonic $(0,q)$-forms with values in $L^k \otimes E$ on $X$. For simplifying the notations, sometimes we will denote $\ddbar$, $\ddbar^*$ and $\square$. For forms with values in $L^k\otimes E$, we denote the Hermitian norm $|\cdot|:=|\cdot|_{h_k,\omega}$  induced by $\omega,h^L, h^E$ and the $L^2$-inner product $\|\cdot\|:=\|\cdot\|_{L_{0,q}^2(X,L^k\otimes E)}$ for each $q\in \N$.  Let $\{ w_j \}_{j=1}^n$ be a local orthonormal frame of $T^{(1,0)}X$ with respect to $\omega$ with dual frame $\{w^j \}$ of $T^{(1,0)*}X$.
    
  Let 
  $\cH^{0,q}(X,L^k\otimes E)$
  be the space of harmonic $(0,q)$-forms with values in $L^k\otimes E$. Let $\{s^k_j\}_{j\geq1}$ be an orthonormal basis and denote by $B_k^q$
  the Bergman density function defined by
  \begin{equation}\label{e:Bergfcn}
  B_k^q(x)=\sum_{j\geq 1}|s^k_j(x)|_{h_k,\omega}^2\,,
  \;x\in X,
  \end{equation}
  where $|\cdot|_{h_k,\omega}$ is the pointwise norm of a form. The function \eqref{e:Bergfcn} 
  is well-defined by an adaptation of 
  \cite[Lemma 3.1]{CM} to form case.  
  
  We follow the notations in \cite[Section 3.2]{Wh:16} and show the submeanvalue formulas of harmonic forms in $\cH^{n,q}(X,L^k\otimes E)$.
  Let $(L,h^L)$ and $(E,h^E)$ be Hermitian holomorphic line bundles over $X$. For any compact subset $K$ in $X$, the interior of $K$ is denoted by $\mathring{K}$. Let $K_1, K_2$ be compact subsets in $X$, such that $K_1\subset \mathring{K_2}$. Then there exists a constant $c_0=c_0(\omega, K_1, K_2)>0$ such that for any $x_0\in K_1$, the holomorphic normal coordinate around $x_0$ is $V\cong W\subset \C^n$, where $$W:=B(c_0):=\{ z\in\C^n: |z|<c_0\}, \quad V:=B(x_0,c_0)\subset \mathring{K_2}\subset K_2,$$ $z(x_0)=0$, and
  $\omega(z)= \sqrt{-1} \sum_{i,j} h_{ij}(z)d{z_i} \wedge d{\overline{z}_j}$ with $h_{ij}(0)=\frac{1}{2}\delta_{ij}$.
   \begin{lemma}\label{keylemma}
   	Let $(X,\omega)$ be a Hermitian manifold of dimension $n$ and
   	$(L,h^L)$ and $(E,h^E)$ be holomorphic Hermitian line bundles over $X$. 
   	Let $K_1$ and $K_2$ be compact subsets in $X$ such that $K_1\subset \mathring{K_2}$. Let $1\leq q\leq n$. Assume $(L,h^L)$ satisfies (\ref{hypo_semi}) for $ x\in\mathring{K_2}$. 
   	Then,
   	\begin{itemize}
   		\item[(1)]
   		there exists a constant $C>0$ such that
   		\begin{equation}\label{eq4} 
   		\int_{|z|<r}|\alpha|_{h_k,\omega}^2 dv_X \leq Cr^{2q}\int_{X}|\alpha|_{h_k,\omega}^2 dv_X
   		\end{equation}
   		for any $\alpha\in \cH^{n,q}(X,L^k\otimes E)$ and $0<r<\frac{C_0}{2^n}$;
   		\item[(2)]
   		there exists a constant $C>0$ such that 
   		\begin{equation}\label{eq161}
   		|\alpha(x_0)|_{h_k,\omega}^2\leq Ck^n \int_{|z|<\frac{2}{\sqrt{k}}}|\alpha|_{h_k,\omega}^2 dv_X
   		\end{equation}
   		for any $x_0\in K_1$, $\alpha\in\cH^{n,q}(X,L^k\otimes E)$ and $k$ sufficiently large, 
   	\end{itemize}
   	where $|\cdot|_{h_{k},\omega}^2$ is the pointwise Hermitian norm induced by $\omega$, $h^L$ and $h^E$.
   \end{lemma}
  \begin{proof}
  	In \cite[Lemma 3.4, 3.5]{Wh:16} the assertion was proved for all $1\leq q\leq n$ for a semipositive line bundle on $\mathring{K_2}$. However, in order to prove the assertion for a fixed $q$, it is enough to assume $(L,h^L)$ is Nakano $q$-semipositive, i.e., it satisfies (\ref{hypo_semi}) for $ x\in\mathring{K_2}$. 
  	Indeed, if $(L,h^L)$ satisfies (\ref{hypo_semi}) for $ x\in\mathring{K_2}$, we have
  	\be
  	c_1(L, h^L)\wedge T_{\alpha}\wedge \omega_{q-1}=(2\pi)^{-1}\langle[\sqrt{-1} R^L, \Lambda] \alpha, \alpha \rangle_h \omega_n\geq 0
  	\ee
  	on $\mathring{K_2}$. Thus, the inequality in
  	\cite[(3.11)]{Wh:16},
  $	i\dbar\ddbar(T_{\alpha}\wedge \omega_{q-1})
  	\geq -C_4|\alpha|_h^{2}\omega_n$,
  	 still holds for $\alpha\in \cH^{n,q}(X,L^k\otimes E)$ and the rest part of the proof is unchanged. Thus this sub-meanvalue proposition analogue to \cite[Lemma 3.4, Lemma 3.5]{Wh:16} follows.
  \end{proof}
  
 Analogue to \cite{BB:02,Wh:16,Wh:17}, we obtain a local estimates for the Bergman density functions as follows.  
  \begin{thm}\label{Localmain}
  	Let $(X,\omega)$ be a Hermitian manifold of dimension $n$ and let
  	$(L,h^L)$ and $(E,h^E)$ be holomorphic Hermitian line bundles over $X$, and $1\leq q\leq n$. 
  	Let $K\subset X$ be a compact subset and $(L,h^L)$ is Nakano $q$-semipositive with respect to $\omega$ on a neighborhood of $K$.
  	Then there exists  $C>0$ depending on $K$, $\omega$, $(L,h^L)$ and $(E,h^E)$, such that
  	\begin{equation}
  		B^q_k(x) \leq Ck^{n-q}\quad \mbox{for all}~ x\in K, k\geq 1,
  	\end{equation}
  	where $B^q_k(x)$ is defined by \eqref{e:Bergfcn} for harmonic 
  	$(0,q)$-forms with values in $L^k\otimes E$. 
  \end{thm}
  
  \begin{proof}
  	 We repeat the procedure in the proof of \cite[Theorem 1.1]{Wh:16} by using Lemma \ref{keylemma} instead of \cite[Lemma 3.4, 3.5]{Wh:16}. 
  	 By combine (\ref{eq161}) and the case $r=\frac{2}{\sqrt{k}}$ of (\ref{eq4}), we have 
  	 there exists $C>0$ such that 
  	 \begin{equation}\label{eq162}
  	 S_k^q(x):=\sup\left\{ \frac{|\alpha(x)|_{h_k,\omega}^2}{\|\alpha\|^{2}_{L^2}}: 
  	 \alpha\in \cH^{n,q}(X,L^{ k}\otimes E)  \right\} \leq C k^{n-q}
  	 \end{equation}
  	 for any $x\in K_1$ and $k\geq 1$. Finally, it follows from the fact $S^q_k(x)\leq B^q_k(x)\leq C S^q_k(x)$ and replacing $E\bigotimes\Lambda^n (T^{(1,0)}X)$ for $E$ in $\cH^{n,q}(X,L^k\otimes E)$.
  \end{proof}
  \begin{proof}[Proof of Theorem \ref{thm_local}]
  	Combining Theorem \ref{Localmain} and Proposition \ref{prop_refine_semip}.
  \end{proof}
  \subsection{The growth of cohomology on coverings}
\begin{proof}[Proof of Corollary \ref{thm_covering}]
	For a fundamental domain $U\Subset X$ with respect to $\Gamma$, by Theorem \ref{thm_local} and Proposition \ref{prop_refine_semip},
$	\dim_{\Gamma}\cH^{0,j}(X,L^k\otimes E) 
	=\int_{U}B^j_k(x) d v_{X}\leq Ck^{n-j}$ for all $j\geq q$.
\end{proof}

\begin{cor}\label{cor_cover_w}
	Let $(X,\omega)$ be a $\Gamma$-covering Hermitian manifold of dimension $n$.
	Let $(L,h^L)$ and $(E,h^E)$ be two $\Gamma$-invariant 
	holomorphic Hermitian line bundles on $X$.
	\begin{itemize}
		\item[(1)] If $\tau(L,h^L,\omega)\geq 0$ on $X$,
		then there exists $C>0$ such that for any 
		$k\geq 1$,
		\begin{equation}
			\dim_{\Gamma}{\overline{H}}^{0,n}_{(2)}(X, L^k\otimes E)
			\leq C.
		\end{equation}
		
		\item[(2)] If $\tau(L,h^L,\omega)\leq 0$ on $X$,
		then there exists $C>0$ such that for any 
		$k\geq 1$,
		\begin{equation}
			\dim_{\Gamma}{\overline{H}}^{0,0}_{(2)}(X, L^k\otimes E)
			\leq C. 
		\end{equation} 
	\end{itemize}
\end{cor}    
\begin{proof}
Apply Proposition \ref{prop_tr_pos}(2)(4), Corollary \ref{thm_covering} and Serre duality \cite[6.3.15]{CD:01}.
\end{proof}

	Since connected complex manifolds are either compact or $n$-complete, see \cite[IX.(3.5)]{Dem}, we can rephrase Corollary \ref{cor_cover_w} for the trivial $\Gamma$ by (\ref{dual_convx}) as follows. 
	Let $(X,\omega)$ be a connected Hermitian manifold of dimension $n$. Let $(L,h^L)$ and $(E,h^E)$ be holomorphic Hermitian line bundles on $X$.	
	If $\tau(L,h^L,\omega)\geq 0$ on $X$, then
		$\dim H^{n}(X, L^k\otimes E)
		\leq C$ for any 
		$k\geq 1$;	
		if $\tau(L,h^L,\omega)\leq 0$ on $X$, then
	     $\dim [H^{0,0}(X, L^k\otimes E)]_0
		\leq C$ for any 
		$k\geq 1$.
  
\subsection{The growth of cohomology on general Hermitian manifolds}
As another application, we can refine the main result in \cite{Wh:17}.
 \begin{proof}[Proof of Theorem \ref{thm_L2general}]
By Theorem \ref{Localmain} and the concentration condition, we have
\begin{eqnarray}\nonumber
\dim \ov H^{0,q}_{(2)}(X,L^k\otimes E)
&=&\dim \cH^{0,q}(X,L^k\otimes E)\\
&=&\sum_{j\geq 1} \|s^k_j\|^2\leq C_0\int_K B^q_k(x)dv_X\leq C_0Ck^{n-q}\vol(K)
\end{eqnarray}
for sufficiently large $k$. Note that $H^{0,q}_{(2)}(X,F)= \ov H^{0,q}_{(2)}(X,F)$ and the dimension is finite, when the fundamental estimate holds in bidegree $(0,q)$ for forms with values in a holomorphic Hermitian vector bundle $(F,h^F)$.
 \end{proof} 
     
\section{Refined estimates on complex manifolds with convexity }\label{sec_pf_main}
\subsection{Proof of the result on $q$-convex manifolds}

Let $X$ be a $q$-convex manifold of dimension $n$. Let $\varrho$ be a exhaustion function of $X$ and $K$ a compact exceptional set in $X$. By  definition, $\varrho\in \cC^\infty(X,\R)$ satisfies  $X_c:=\{ \varrho<c\}\Subset X$ for all $c\in \R$,  $\sqrt{-1}\dbar\ddbar\varrho$ has $n-q+1$ positive eigenvalues on $X\setminus K$. In this section, we fix real numbers $u_0, u$ and $v$ satisfying $u_0<u<c<v$ and $K\subset X_{u_0}$.  

We outline the idea of our proof of Theorem \ref{thm_main}. Let $(L,h^L), (E,h^E)$ be holomorphic Hermitian line bundles on $X$. The fundamental estimate holds in bidegree $(0,j)$ for forms with values in $L^k\otimes E$ for large $k$ and each $q\leq j\leq n$ on $X_c$ when $X$ is a $q$-convex manifold, see Proposition \ref{1coxFE}, which was obtained in \cite[Theorem 3.5.8]{MM} for the proof of Morse inequalities on $q$-convex manifolds. We observe that the Nakano $q$-semipositive is preserved by the modification of $h^L$, see Proposition \ref{prop_mod_chi}. By Theorem \ref{thm_L2general}, Proposition \ref{P:semipos} and related cohomology isomorphism, we obtain the desired results for $j\geq q$.
 
Firstly, we choose now a Hermitian metric $\omega$ on $X$ from \cite[Lemma 3.5.3]{MM}. 

\begin{lemma} \label{lowbd_rho_lem}
	For any $C_1>0$ there exists a metric $g^{TX}$ (with Hermitian form $\omega$) on $X$ such that for any $j\geq q$ and any holomorphic Hermitian vector bundle $(F,h^F)$ on $X$,
	\begin{equation}
	\left\langle (\dbar\ddbar\varrho)(w_l,\ol{w}_k)\ol{w}^k\wedge i_{\ol{w}_l}s,s \right\rangle_h\geq C_1|s|^2, \quad s\in \Omega^{0,j}_0(X_{v}\setminus \ov X_{u_0},F),
	\end{equation}
	where $\{ w_l \}_{l=1}^n$ is a local orthonormal frame of $T^{(1,0)}X$ with dual frame $\{ w^l\}_{l=1}^n$ of $T^{(1,0)*}X$.
\end{lemma}

Now we consider the $q$-convex manifold $X$ associated with the metric $\omega$ obtained above as a Hermitian manifold $(X,\om)$. Note for arbitrary holomorphic vector bundle $F$ on a relatively compact domain $M$ in $X$, the Hilbert space adjoint $\ddbar_H^{F*}$ of $\ddbar^F$ coincides with the formal adjoint $\ddbar^{F*}$ of $\ddbar^F$ on $B^{0,j}(M,F)=\Dom(\ddbar_H^{F*})\cap \Omega^{0,j}(\overline{M},F)$, $1\leq j\leq n$. So we simply use the notion $\ddbar^{F*}$ on $B^{0,j}(M,F)$, $1\leq j\leq n$.

Secondly, we will modify hermitian metric $h^L_\chi$ on $L$ and show the fundamental estimate fulfilled.
Let $\chi(t)\in\cC^\infty(\R)$ such that $\chi'(t)\geq 0$, $\chi''(t)\geq 0$, which will be determined later. We define a Hermitian metric $h^{L}_\chi:=h^{L}e^{-\chi(\varrho)}$ on $L$, and thus the modified curvature is
\begin{equation}
R^{L_\chi}=R^L+\chi'(\varrho)\dbar\ddbar\varrho+\chi''(\varrho)\dbar\varrho\wedge\ddbar\varrho.
\end{equation}

\begin{prop}\label{1coxFE}
	Let $X$ be a $q$-convex manifold of dimension $n$ with the exceptional set $K\subset X_c$. Then there exists a compact subset $K'\subset X_c$ and $C_0, C_3>0$ such that for sufficiently large $k$, we have 
	\begin{equation}
	\|s\|^2\leq \frac{C_0}{k}(\|\ddbar^E_ks\|^2+\|\ddbar^{E*}_{k,H}s\|^2)+C_0\int_{K'} |s|^2 dv_X
	\end{equation}
	for any $s\in \Dom(\ddbar^E_k)\cap \Dom(\ddbar^{E*}_{k,H})\cap L^2_{0,j}(X_c,L^k\otimes E)$ and $q\leq j \leq n$,
	where $\chi'(\varrho)\geq C_3$ on $X_v\setminus \overline{X}_u$ and the $L^2$-norm is given by $\omega$, $h^{L^k}_\chi$ and $h^E$ on $X_c$.
\end{prop} 
\begin{proof}
	See \cite[Proposition 3.8]{Wh:17} or \cite[Theorem 3.5.8]{MM}.
\end{proof}  

Thirdly, we will show that $(L_\chi, h^{L_\chi})$ preserves the certain semi-positivity of $(L, h^L)$ by choosing a appropriate $\chi$ as follows. Let $C_3>0$ be in Lemma \ref{1coxFE}. We choose $\chi\in\cC^\infty(\R)$ such that $\chi''(t)\geq 0$, $\chi'(t)\geq C_3$ on $(u,v)$ and $\chi(t)=0$ on $(-\infty,u_0)$. Therefore, $\chi'(\varrho(x))\geq C_3>0$ on $X_v\setminus \overline{X}_u$ and $\chi(\varrho(x))=\chi'(\varrho(x))=0$ on $X_{u_0}$. Note $K\subset X_{u_0}$ and $u_0<u<c<v$. Now we have a fixed $\chi$ which leads to the following proposition.

\begin{prop}\label{prop_mod_chi}
	 $X$ is a $q$-convex manifold with Hermitian metric $\omega$ given by Lemma \ref{lowbd_rho_lem}. Let $j\geq q$. Suppose $(L,h^L)$ satisfies 
	\be
	\left\langle[\sqrt{-1}R^{(L,h^{L})},\Lambda]\alpha,\alpha\right
	\rangle_h \geq 0\quad \mbox{for all}~ \alpha\in \wedge^{n,j}T_x^*X, x\in X_c.
	\ee
Then,
	$(L_\chi,h^{L_\chi})$ satisfies
	\be
	\left\langle[\sqrt{-1}R^{(L_\chi,h^{L_\chi})},\Lambda]\alpha,\alpha\right\rangle_h \geq 0\quad \mbox{for all}~ \alpha\in \wedge^{n,j}T_x^*X, x\in X_c.
	\ee
	In particular, if $(L,h^L)\geq 0$ on $X_c$, $(L_\chi,h^{L_\chi})$ satisfies
	\be
	\left\langle[\sqrt{-1}R^{(L_\chi,h^{L_\chi})},\Lambda]\alpha,\alpha\right\rangle_h \geq 0\quad \mbox{for all}~ \alpha\in \wedge^{n,j}T_x^*X, x\in X_c, j\geq q.
	\ee
\end{prop}

\begin{proof}
	 $\imat R^{L_\chi}=\imat R^L+\imat \chi'(\varrho)\dbar\ddbar\varrho+\imat \chi''(\varrho)\dbar\varrho\wedge\ddbar\varrho$ on $X_c$.
		From the above definition of $\chi$, we have $\chi'(\varrho)\geq 0$ on $X$, $\chi'(\varrho)=0$ on $\ov X_{u_0}$, and $\chi''(\varrho)\geq 0$ on $X$. Since  $\imat\dbar\varrho\wedge\ddbar\varrho\geq 0$ on $X_c$, we have $\imat \chi''(\varrho)\dbar\varrho\wedge\ddbar\varrho\geq 0$ on $X_c$. Therefore, we only need to show that, for all $ \alpha\in \wedge^{n,j}T_x^*X$, $x\in X_c\setminus \ov X_{u_0}$,
		\be
		\left\langle [\sqrt{-1}\dbar\ddbar\varrho,\Lambda] \alpha,\alpha \right\rangle_h\geq 0.
		\ee  
	In fact, from Lemma \ref{lowbd_rho_lem}, for $s\in \Omega_0^{n,j}(X_v\setminus \ov X_{u_0})=\Omega_0^{0,j}(X_v\setminus \ov X_{u_0}, K_X)$ with $s(x)=\alpha\in \wedge^{n,j}T_x^*X$, $x\in X_c\setminus \ov X_{u_0}$,
	\begin{eqnarray}\nonumber
\left\langle [\sqrt{-1}\dbar\ddbar\varrho,\Lambda] \alpha,\alpha \right\rangle_h
&=&\left\langle [\sqrt{-1}\dbar\ddbar\varrho,\Lambda] s,s \right\rangle_h(x)
=
\left\langle \sqrt{-1}\dbar\ddbar\varrho\wedge\Lambda s,s \right\rangle_h(x)\\
&=& 	 
\left\langle   (\dbar\ddbar\varrho)(w_l,\ol{w}_k)\ol{w}^k\wedge i_{\ol{w}_l}s,s \right\rangle_h(x)
\geq C_1|s|_h^2(x)= C_1|\alpha|_h^2\geq 0.
	\end{eqnarray}
	Thus the proof is complete.
\end{proof}

Now we combine the above components and obtain:
\begin{thm}\label{T:qconvex}
	Let $X$ be a $q$-convex manifold of dimension $n$ with a Hermitian metric $\omega$ given by Lemma \ref{lowbd_rho_lem}.  Let $(L,h^L)$ and $(E,h^E)$ be holomorphic Hermitian line bundles on $X$. Let the exceptional set $K\subset X_c$. Let $j\geq q$ and $(L,h^L)$ satisfies, with respect to $\omega$, 
	\be
	\left\langle[\sqrt{-1}R^{(L,h^{L})},\Lambda]\alpha,\alpha\right\rangle_h \geq 0\quad \mbox{for all}~ \alpha\in \wedge^{n,j}T_x^*X, x\in X_c.
	\ee
	Then, for all $k\geq 1$,
$	\dim H^{j}(X,L^k\otimes E)\leq Ck^{n-j}$.
\end{thm}
 
\begin{proof}
Proposition \ref{1coxFE} entails the fundamental estimate holds in bidegree $(0,j)$ for forms with values in $L^k\otimes E$ for large $k$ on $X_c$ with respect to $\omega, h^L_\chi$ and $h^E$ and $j\geq q$. Thus, by Proposition \ref{prop_mod_chi} and Theorem \ref{thm_L2general}, there exists $C>0$ such that for sufficiently large $k$,
\begin{equation}
\dim H_{(2)}^{0,j}(X_c,L^k\otimes E)=\dim \cH^{0,j}(X_c,L^k\otimes E)\leq Ck^{n-j}
\end{equation}
holds with respect to $h^E$ and the chosen metrics $\omega$ and $h^{L}_\chi$ on $X_c$ (as in \cite{Wh:17}).
By results of H\"{o}rmander \cite[Theorem 3.5.6]{MM}, Andreotti-Grauert \cite[Theorem 3.5.7]{MM} and the Dolbeault isomorphism \cite[Theorem B.4.4]{MM}, we have, for $j\geq q$,
\be
\quad\quad H^j(X,L^k\otimes E)\cong H^j(X_v,L^k\otimes E)\cong H^{0,j}(X_v,L^k\otimes E)\cong H_{(2)}^{0,j}(X_c,L^k\otimes E).
\ee
Thus the conclusion holds for sufficiently large $k$. Note that for any holomorphic vector bundle $F$, $\dim H^j(X,F)<\infty$ for $j\geq q$ by the result of Andreotti-Grauert \cite[Theorem B.4.8]{MM}. So the conclusion holds for all $k\geq 1$.
\end{proof}
   
\begin{proof}[Proof of Theorem \ref{thm_main}]
	Let $X_c$ be a sublevel set including $K$ such that $(L,h^L)\geq 0$ on $X_c$. From Proposition \ref{P:semipos}, $(L,h^L)\geq 0$ on $X_c$ implies for any Hermitian metric $\omega$, 
	\be
	\left\langle[\sqrt{-1}R^{(L,h^{L})},\Lambda]\alpha,\alpha\right\rangle_h \geq 0\quad \mbox{for all}~ \alpha\in \wedge^{n,j}T_x^*X, x\in X_c, j\geq 1.
	\ee
 Then the conclusion follows by Theorem \ref{T:qconvex}. 
\end{proof}

By adapting the duality formula \cite[20.7 Theorem]{HL:88} to Theorem \ref{thm_main}, we have the analogue result to \cite[Remark 4.4]{Wh:16} for seminegative line bundles.
\begin{cor}
	Let $X$ be a $q$-convex manifold of dimension $n$ and let $(L,h^L), (E,h^E)$ be holomorphic Hermitian line bundles on $X$. Let $(L,h^L)$ be seminegative on a neighborhood of the exceptional subset $K$. Then there exists $C>0$ such that for any $0\leq j\leq n-q$ and  $k\geq 1$, the $j$-th cohomology with compact supports
	\begin{equation}
	\dim [H^{0,j}(X,L^k\otimes E)]_0\leq Ck^{j}.
	\end{equation}
\end{cor}
\begin{proof}
	For any $ q\leq s\leq n$,
	$\dim [H^{0,n-s}(X,L^k\otimes E)]_0= \dim H^{0,s}(X,L^{k*}\otimes E^*\otimes K_X)\leq Ck^{n-s}$ by  Theorem \ref{thm_main} and (\ref{dual_convx}), see \cite{AG:62} and \cite[20.7 Theorem]{HL:88}.
\end{proof} 

\begin{rem}[Vanishing theorems on $q$-convex manifolds]
	 \label{remqconvex}
	Let $(E,h^E)$ be a holomorphic vector bundle on $X$. If $(L,h^L)>0$ on $X_c$ with $K\subset X_c$ instead of the hypothesis $(L,h^L)\geq 0$ on $X_c$ in Theorem \ref{thm_main}, then for $j\geq q$ and sufficiently large $k$,
	$\dim H^j(X,L^k\otimes E)=0$,
see \cite[Theorem 3.5.9]{MM}. And it can be generalized to Nakano $q$-positive as follows.
\end{rem}   

\begin{thm}\label{T:qvsh}
	Let $(X,\omega)$ be a $q$-convex manifold of dimension $n$ with the Hermitian metric $\omega$ given by Lemma \ref{lowbd_rho_lem} and $1\leq q\leq n$. Let $E, L$ be holomorphic vector bundle with $\rank(L)=1$. 
	Let $K\subset X$ be the exceptional set. If $(L,h^L)$ is Nakano $p$-positive with respect to $\omega$ on $X_c$  with $K\subset X_c$, then for $j\geq \max\{p,q\}$ and $k$ sufficiently large,
	\begin{equation}  
	H^j(X,L^k\otimes E)=0.
	\end{equation}
\end{thm}	
\begin{proof} 
	We can shrink $X_c$ with $K\subset X_c$ such that $(L,h^L)$ is $p$-positive with respect to $\omega$ on the closure $\ov X_c$. By (\ref{D:q-semi-local}), there exists $C_L>0$ such that
	\be
	\langle R^L(w_i, \ov w_j) \ov w^j \wedge i_{\ov w_i} s,s \rangle _h\geq C_L |s|^2_h
	\ee
	for any $s\in B^{0,j}(X_c, F)$ with arbitrary holomorphic line bundle $F$ and $j\geq p$.
	Thus there exists $C_2>0$, for each $s\in B^{0,j}(X_c,L^k\otimes E)$ with $j\geq \max\{p,q\}$ and $k$ sufficiently large,
	\begin{equation}
	\|s\|^2\leq \frac{C_2}{k}( \|\ddbar^E_k s\|^2+\|\ddbar^{E*}_k s\|^2 )
	\end{equation}   
	 holds with respect to $h^L$ and $\omega$ as in \cite[Lemma 3.5.4]{MM}, and thus it holds for $s\in \cH^{0,j}(X_c,L^k\otimes E)$ with $j\geq \max\{p,q\}$. Then, for $k$ sufficiently large, 
	$H^j(X,L^k\otimes E)\cong \cH^{0,j}(X_c,L^k\otimes E)=0$ with $j\geq \max\{p,q\}$.
\end{proof}    

\begin{rem}[Complex spaces]
	Let $X$ be a $j$-convex K\"{a}hler manifold with $\dim X=n$ and 
	$1\leq j\leq n$. Let $(L,h^L)$ be a holomorphic Hermitian line bundle and 
	$(L,h^L)\geq 0$ on $X$. Let $S$ be a complex space and $f: X\rightarrow S$ 
	a proper surjective holomorphic map. Then, by Theorem \ref{thm_main} and \cite{Mats:16},
$\dim H^p(S, R^qf_*(K_X\otimes L^k))= O(k^{n-p-q})$ for all $(p,q)$ with $p+q\geq j$,  
		where $R^qf_*(\cdot)$ is the $q$-th higher direct image sheaf. 
\end{rem}

\subsection{Pseudoconvex, weakly $1$-complete, and complete manifolds}
  
Analogue to the case of $q$-convex manifolds, we can generalize other results in \cite{Wh:17} as follows. Holomorphic Morse inequalities for weakly $1$-complete manifolds and pseudoconvex domain were obtained in \cite{M:92} and \cite[Theorem 3.5.10, 3.5.12]{MM}.
\begin{thm}\label{pcoxthm}
	Let $M\Subset X$ be a smooth (weakly) pseudoconvex domain in a complex manifold $X$ of dimension $n$. Let $\omega$ be a Hermitian metric on $X$. Let $(L,h^L)$ and $(E,h^E)$ be holomorphic Hermitian line bundles on $X$.  
	Let $1\leq q\leq n$. Assume $(L,h^L)$ is Nakano $q$-semipositive with respect to $\omega$ on $M$, and $(L,h^L)$ is Nakano $q$-positive with respect to $\omega$ in a neighbourhood of $bM$.
	Then there exists $C>0$ such that for sufficiently large $k$, we have
	\begin{equation}
	\dim H^{0,j}_{(2)}(M,L^k\otimes E)\leq Ck^{n-j}\quad \mbox{for all}~ q\leq j\leq n.
	\end{equation} 
\end{thm} 
\begin{proof}
	We follow \cite[Theorem 1.5, (3.29)]{Wh:17} and \cite[Theorem 3.5.10]{MM}.	
Let $\rho\in \cC^\infty(X,\R)$ be a defining function of $M$ such that $M=\{x\in X:\rho(x)< 0  \}$ with $|d\rho|=1$ on the boundary $bM$. Let $x\in bM$.  For $s\in \Omega^{0,\bullet}(\ov M,L^k\otimes E)$,
the Levi form defined by $ \cL_\rho(s,s)(x)
:=\sum_{j,k=2}^n(\dbar\ddbar\rho)(w_k,\ov{w}_j)\langle \ov{w}^j\wedge i_{\ov{w}_k}s(x),s(x)\rangle_h$. Since $M$ is pseudoconvex, it follows that, for $s\in B^{0,q}( M,L^k\otimes E)$,
\begin{equation}
\int_{bM} \cL_\rho(s,s)dv_{bM}\geq 0.
\end{equation}
Let $X_c:=\{ x\in X:\rho(x)<c \}$ for $c\in \R$.
We fix $u<0<v$ such that $L$ is Nakano $q$-positive with respect to $\omega$ on a open neighbourhood of $X_v\setminus \ov{X}_u$, then there exists $C_L>0$ such that for any holomorphic Hermitian vector bundle $(F,h^F)$ on $X$,
\begin{equation}
\langle R^L(w_l,\ol{w}_k)\ol{w}^k\wedge i_{\ol{w}_l}s,s \rangle_h\geq C_L|s|^2, \quad s\in \Omega^{0,q}_0(X_v\setminus \ov{X}_u,F).
\end{equation}
 
By the Bochner-Kodaira-Nakano formula with boundary term \cite[Corollary 1.4.22]{MM}, there exist $C_4\geq 0$ and $C_5\geq 0$ such that for any $s\in B^{0,q}(M,L^k\otimes E)$ with $\supp (s)\in X_v\setminus\ov{X}_u$,
\begin{eqnarray}\nonumber
\frac{3}{2}(\|\dbar^E_k s\|^2+\|\dbar^{E*}_k s\|^2)
&\geq& \langle  R^{L^k\otimes E\otimes K^*_X}(w_j,\ov{w}_k)\ov{w}^k\wedge i_{\ov{w}_j} s,s\rangle+\int_{bM}\cL_{\rho}(s,s)dv_{bM}-C_4\|s\|^2\\
&\geq & \int_M(kC_L-C_5-C_4)|s|^2dv_X.
\end{eqnarray}
For any $k\geq k_0:=[2\frac{C_4+C_5}{C_L}]+1$, we have $C_L-\frac{C_4+C_5}{k}\geq \frac{1}{2}C_L$. Let $C_2:=\frac{3}{C_L}$. For any $s\in B^{0,q}(M,L^k\otimes E)$ with $\supp(s)\subset X_v\setminus \overline{X}_u$ and $k\geq k_0>0$, we have
\begin{equation}
\|s\|^2\leq \frac{C_2}{k}( \|\ddbar^E_k s\|^2+\|\ddbar^{E*}_k s\|^2 )
\end{equation}   
where the $L^2$-norm $\|\cdot\|$ is given by $\omega$, $h^{L^k}$ and $h^E$ on $M$. 

Note the fact that $B^{0,q}(M,L^k\otimes E)$ is dense in $\Dom(\ddbar^E_k)\cap \Dom(\ddbar^{E*}_{k,H})\cap L^2_{0,q}(M,L^k\otimes E)$ with respect to the graph norm of $\ddbar^E_k+\ddbar^{E*}_{k,H}$.
Following the same argument in Lemma \ref{1coxFE} (without the modification of $h^L$ by $\chi$), we conclude that
there exist a compact subset $K'\subset M$ and $C_0>0$ such that for sufficiently large $k$, we have 
\begin{equation}
\|s\|^2\leq \frac{C_0}{k}(\|\ddbar^E_ks\|^2+\|\ddbar^{E*}_{k,H} s\|^2)+C_0\int_{K'} |s|^2 dv_X
\end{equation}
for any $s\in \Dom(\ddbar^E_k)\cap \Dom(\ddbar^{E*}_{k,H})\cap L^2_{0,q}(M,L^k\otimes E)$,
where the $L^2$-norm is given by $\omega$, $h^{L^k}$ and $h^E$ on $M$. 
 That is, the fundamental estimate holds in bidegree $(0,q)$ for forms with values in $L^k\otimes E$ for large $k$. Finally, we apply Theorem \ref{thm_L2general} and Proposition \ref{prop_refine_semip}.
\end{proof}
The polynomial growth of dimension of cohomology of Griffiths $q$-positive 
line bundles on weakly $1$-complete manifolds via holomorphic Morse inequalities, we
refer to \cite{M:92}.
For the Nakano $q$-positive cases, by applying Theorem \ref{pcoxthm} as in 
\cite[Proof of Theorem 1.6]{Wh:17}, we obtain:

\begin{cor}\label{weak1thm}
	Let $X$ be a weakly $1$-complete manifold of dimension $n$ with 
	a smooth plurisubharmonic exhaustion function $\rho$ and $\omega$ 
	be a Hermitian metric on $X$. Let $(L,h^L)$ and $(E,h^E)$ be holomorphic 
	Hermitian line bundles on $X$.
	Let $1\leq q\leq n$ and $(L,h^L)$ is Nakano $q$-semipositive with respect to $\omega$ on $X$.

	$(1)$ Assume $(L,h^L)$ is Nakano $q$-positive with respect to 
	$\omega$ on $X\setminus K$ for a compact subset $K$.
	Then, for any sublevel set $X_c:=\{ \rho<c \}$ with smooth boundary 
	and $K\subset X_c$, there exists $C>0$ such that for $k$ sufficiently large,
	\begin{equation}
	\dim H^{0,j}_{(2)}(X_c,L^k\otimes E)\leq Ck^{n-j} \quad\mbox{for all}~ q\leq j\leq n. 
	\end{equation}  
	
	$(2)$ Assume $(L,h^L)$ is positive on $X\setminus K$ for a compact subset $K$.
	Then there exists $C>0$ such that for $k$ sufficiently large,
	\begin{equation}
	\dim H^j(X,L^k\otimes E)\leq Ck^{n-j} \quad\mbox{for all}~ q\leq j\leq n. 
	\end{equation}  
\end{cor} 

\begin{proof}
	(1) is from $X_c$ is a smooth pseudoconvex domain and Theorem \ref{pcoxthm}; (2) follows from (1) and 
	$H^j(X,L^k\otimes E)\cong H^{0,j}_{(2)}(X_c,L^k\otimes E)$ for all $j\geq q$ and sufficiently large $k$.
\end{proof}

Similarly, we also can refine \cite[Theorem 1.2]{Wh:17} on complete manifolds.
\begin{thm}\label{completethm}
	Let $(X,\omega)$ be a complete Hermitian manifold of dimension $n$. Let $(L,h^L)$ be a holomorphic Hermitian line bundle on $X$. Assume there exists a compact subset $K\subset X$ such that $\sqrt{-1}R^{(L,h^L)}=\omega$ on $X\setminus K$.
	Let $1\leq q\leq n$ and $(L,h^L)$ is Nakano $q$-semipositive with respect to $\omega$ on $K$.
	Then there exists $C>0$ such that for sufficiently large $k$, we have 
	\begin{equation}
	\dim H_{(2)}^{0,j}(X,L^k\otimes K_X)\leq Ck^{n-j}\quad \mbox{for all}~ q\leq j\leq n.
	\end{equation}
\end{thm}

\begin{proof}
	Since $(X,\omega)$ is complete, $\ddbar^{E*}_{k,H}=\ddbar^{E*}_k$ for arbitrary holomorphic Hermitian vector bundle $(E,h^E)$. In a local orthonormal frame $\{ \omega_j \}_{j=1}^n$ of $T^{(1,0)}X$ with dual frame $\{ w^j\}_{j=1}^n$ of $T^{(1,0)*}X$, $\omega=\sqrt{-1}\sum_{j=1}^n \omega^j\wedge\ov\omega^j$ and $\Lambda=-\sqrt{-1}i_{\ov w_j}i_{w_j}$. Thus $\sqrt{-1}R^{(L,h^L)}=\sqrt{-1}\sum_{j=1}^n \omega^j\wedge\ov\omega^j$ outside $K$. Let $\{e_k\}$ be a local frame of $L^k$. For $s\in \Omega^{n,q}_0(X\setminus K,L^k)$, we can write $s=\sum_{|J|=q} s_J\omega^1\wedge\cdots\wedge\omega^n\wedge\ov\omega^J\otimes e_k$ locally, thus
	\begin{equation}
		[\sqrt{-1}R^L,\Lambda]s
		=\sum_{|J|=q}(q s_J\omega^1\wedge\cdots\wedge\omega^n\wedge\ov\omega^J)\otimes e_k 
		= qs.
	\end{equation}
	Since $(X\setminus K, \sqrt{-1}R^{(L,h^L)})$ is K\"{a}hler, we apply Nakano's inequality \cite[(1.4.52)]{MM}, 
	\begin{equation}
		\|\ddbar_k s\|^2+\|\ddbar^{*}_k s\|^2 
		\geq k \langle  [\imat R^L,\Lambda]s,s \rangle\geq  qk\|s\|^2\geq k\|s\|^2.
	\end{equation}
	Therefore, we have
	$	\|s\|^2\leq \frac{1}{k}( \|\ddbar_k s\|^2+\|\ddbar^{*}_k s\|^2 )$
	for $s\in \Omega^{n,q}_0(X\setminus K,L^k)$ with $1\leq q\leq n$ with respect to $h^L$ and $\omega$.
	
	Next we follow the analogue argument in \cite[Proposition 3.8]{Wh:17} to obtain the fundamental estimates as follows. Let $V$ and $U$ be open subsets of $X$ such that $K\subset V\Subset U\Subset X$. We choose a function $\xi\in \cC^\infty_0(U,\R)$ such that $0\leq \xi\leq 1$ and $\xi\equiv 1$ on $\ov V$. We set $\phi:=1-\xi$, thus $\phi\in \cC^\infty(X,\R)$ satisfying $0\leq \phi \leq 1$ and $\phi \equiv 0$ on $\ov V$. 
	
	Now let $s\in \Omega_0^{n,q}(X,L^k)$, thus $\phi s\in \Omega^{n,q}_0(X\setminus K, L^k)$. We set $K':=\ov U$, then 
	\begin{equation}
		\|\phi s\|^2\geq \|s\|^2-\int_{K'}|s|^2dv_X,
	\end{equation}
	and similarly there exists a constant $C_1>0$ such that	
	\begin{equation}
		\frac{1}{k}(\|\ddbar_k (\phi s)\|^2+\|\ddbar^{*}_k(\phi s)\|^2)\leq \frac{5}{k}(\|\ddbar_k s\|^2+\|\ddbar^{*}_k s\|^2)+\frac{12C_1}{k}\|s\|^2.
	\end{equation}	
	
	By combining the above three inequalities, there exists $C_0>0$ such that for any $s\in \Omega^{n,q}_0(X,L^k)=\Omega^{0,q}_0(X,L^k\otimes K_X)$ and $k$ large enough
	\begin{equation}   
		\|s\|^2\leq \frac{C_0}{k}(\|\ddbar_ks\|^2+\|\ddbar^{*}_ks\|^2)+C_0\int_{K'} |s|^2 dv_X.
	\end{equation}	
	Finally, since $\Omega_0^{0,\bullet}(X,L^k\otimes K_X)$ is dense in $\Dom(\ddbar_k^{K_X})\cap \Dom(\ddbar_k^{K_X*})$ in the graph-norm, the fundamental estimate holds in bidegree $(0,q)$ for forms with values in $L^k\otimes K_X$ for $k$ large.
	So the conclusion follows from Theorem \ref{thm_L2general} and Proposition \ref{prop_refine_semip}.
\end{proof}

\subsection{Vanishing theorems and the estimate $O(k^{n-q})$} \label{sec_Kahler_vanish}

In this section, we restrict to K\"{a}hler manifolds $X$ and $E=K_X$. Firstly, inspired by \cite{Mat:18JAG,Fujino:12}, we see the injectivity for Nakano $q$-semipositive line bundles.

\begin{lemma}\label{lem_inj}
	Let $(X,\omega)$ be a compact K\"{a}hler manifold of dimension $n$ and let $(L,h^L)$ be holomorphic Hermitian line bundle on $X$. Let $1\leq q\leq n$ and $(L,h^L)$ be Nakano $q$-semipositive with respect to $\omega$ on $X$. Let $s\in H^0(X,L^k)\setminus\{0\}$ for some $k>0$. Then, for every $j\geq q$ and $m\geq 1$, the multiplication map $\cdot\otimes s:$
	\be
	H^j(X,K_X\otimes L^m)\rightarrow H^j(X,K_X\otimes L^{m+k})
	\ee
	is injective. In particular, if $(L,h^L)$ is semipositive, it holds for all $j\geq 1$.
\end{lemma}
\begin{proof}
	We follow \cite[1.5 Enoki's proof]{Fujino:12}. By Proposition \ref{prop_refine_semip} and Hodge theorem, we only need to show the multiplication map $\cdot\otimes s$ between the harmonic spaces
	\be
	\cH^{n,q}(X, L^m)\rightarrow \cH^{n,q}(X, L^{m+k})
	\ee
	is injective for $m\geq 1$. Let $u\in \cH^{n,q}(X, L^m)$. Since $s\in H^0(X,L^k)$, $\ddbar^{L^{m+k}}(s\otimes u)=0$. From the $q$-semipositive and Nakano's inequality \cite[(1.4.51)]{MM},
	$\left\langle[\sqrt{-1}R^{(L,h^L)},\Lambda]u,u\right\rangle_h=0$ on $X$. 
	From \cite[(1.4.44),(1.4.38c)]{MM}, 
	$(\nabla^{L^m})^{1,0*}(s\otimes u)=s\otimes((\nabla^{L^m})^{1,0*}u)=0$. 
	Also we have 
	$\left\langle[\sqrt{-1}R^{L^{m+k}},\Lambda](s\otimes u),(s\otimes u)\right\rangle_h=0$. 
	Thus $$\|\ddbar^{L^{m+k}*}(s\otimes u)\|^2=
	\|(\nabla^{L^{m+k}})^{1,0*}(s\otimes u)\|^2=0.$$ 
	We obtain $s\otimes u\in \cH^{n,q}(X, L^{m+k})$. 
	Suppose $s\otimes u=0$ on $X$. Since $s\neq 0$ and 
	\cite[Ch.VII.3. (2.4) Lemma]{Dem}, $u=0$ on $X$.
\end{proof}

	Let $\kappa(L)$ be the Kodaira dimension of $L$ on a a compact complex manifold $X$  given by
	\be
	\kappa(L)&:=&-\infty, ~\mbox{when}~H^0(X,L^k)=0~\mbox{for all}~ k>0; \mbox{otherwise,}\\
	\kappa(L)&:=&\max\{ m\in \N: \limsup_{k\rightarrow\infty}\frac{\dim H^0(X,L^k)}{k^m}>0 \}\in [0,\dim X].
	\ee	 
By the above lemma and Corollary \ref{thm_covering} with the trivial $\Gamma$, we obtain:

\begin{thm}\label{thm_q_semi_kahler}
	Let $(X,\omega)$ be a compact K\"{a}hler manifold of dimension $n$ and let $(L,h^L)$ be holomorphic Hermitian line bundle on $X$. Let $1\leq q\leq n$ and $(L,h^L)$ be Nakano $q$-semipositive with respect to $\omega$ on $X$. Then, for all $j>\max\{ n-\kappa(L),~q-1 \}$ and $m>0$,
	\be
	H^j(X,K_X\otimes L^m)=0.
	\ee
\end{thm}
\begin{proof}
	We follow \cite[Theorem 4.5]{Mat:18JAG}.
	 Suppose there exist $m>0$ and $j>n-\kappa(L)$ with $j\geq q$ such that $H^j(X,K_X\otimes L^m)\neq 0$. Let $u\in H^j(X,K_X\otimes L^m)\setminus\{0\}$ and let $\{s_j\}_{i=1}^N\subset H^0(X,L^k)$ be linearly independent. By the injectivity Lemma \ref{lem_inj}, $\{ s_i\otimes u \}_{i=1}^N\subset H^j(X,K_X\otimes L^{m+k})$ are linearly independent.
	By Corollary \ref{thm_covering} for compact K\"{a}hler manifolds, we see
	\be
	\frac{\dim H^0(X,L^k)}{k^{\kappa(L)}}\leq \frac{\dim H^j(X,K_X\otimes L^{k+m})}{k^{n-j+1}}\leq \frac{C(k+m)^{n-j}}{k^{n-j+1}}\leq \frac{C}{k}.
	\ee
	By applying $\limsup_{k\rightarrow+\infty}$, there is a contradiction.
\end{proof}

\begin{cor}
	Let $(X,\omega)$ be a compact K\"{a}hler manifold and let $(L,h^L)$ be a holomorphic Hermitian line bundle. If $\mbox{Tr}_\omega R^{(L^*,h^{L^*})}\geq 0$ on $X$ and $\kappa(L^*)>0$, then $\kappa(L)=-\infty$. In particular, if the scalar curvature $r_\omega= 0$ on $X$ and $\kappa(K_X^*)>0$, then $\kappa(K_X)=-\infty$.
\end{cor}
\begin{proof}
$H^0(X,L^m)\cong H^n(X,K_X\otimes L^{m*})=0$ and $r_\omega=2\sum_j Ric(\omega_j,\ov\omega_j)=2\mbox{Tr}_\omega R^{K_X^*}$.
\end{proof}

Secondly, as applications of Bochnar-Kodaira-Nakano formulas, certain Kodaira type vanishing theorems of Nakano $q$-semipositive line bundles hold as follows.
  
\begin{prop}
	Let $(X,\omega)$ be a complete K\"{a}hler manifold of dimension $n$ and $1\leq q\leq n$. Let $(L,h^L)$ be a Nakano $q$-semipositive line bundle with respect to $\omega$ on $X$. Assume there exists $C_0>0$ and a compact subset $K\subsetneqq X$ such that $\sqrt{-1}R^{(L,h^L)}\geq C_0\omega$ on $X\setminus K$. Then, 
	\be
	H_{(2)}^{0,j}(X,K_X\otimes L)=0\quad \mbox{for all}~ j\geq q.
	\ee
\end{prop}

\begin{proof}
 Since $(X,\omega)$ is complete, $\ddbar^{L*}_H=\ddbar^{L*}$.
	For $s\in \Omega^{n,j}_0(X,L)$ for $j\geq q$, from Bochner-Kodaira-Nakano formula, we have
	\be\label{eq_fe}
	\|\ddbar^Ls\|^2+\|\ddbar^{L*}s\|^2\geq\langle [\sqrt{-1}R^L,\Lambda]s,s\rangle_h\geq C_0\|s\|^2_{X\setminus K}=C_0\|s\|^2-C_0\|s\|_K^2.
	\ee
	Since $\Omega^{n,j}_0(X,L)$ is dense in $\Dom(\ddbar^L)\cap\Dom(\ddbar^{L*})$, (\ref{eq_fe}) holds for $s\in \Dom(\ddbar^L)\cap\Dom(\ddbar^{L*})$. Since $K\subsetneqq X$, $s|_{X\setminus K}=0$ for $s\in \cH^{n,j}(X,L)$,  and then $\cH^{n,j}(X,L)=0$. From (\ref{eq_fe}), the fundamental estimate holds for $(0,j)$-form with values in $K_X\otimes L$, and thus 
	$H^{0,j}_{(2)}(X,K_X\otimes L)\cong \cH^{0,j}(X,K_X\otimes L)=0$.
\end{proof}
 
\begin{prop}\label{prop_weak1_vanish}
	Let $(X,\omega)$ be a weakly $1$-complete K\"{a}hler manifold of dimension $n$ and $1\leq q\leq n$. Let $(L,h^L)$ be a Nakano $q$-semipositive line bundle with respect to $\omega$ on $X$. Assume there exist a compact subset $K\subsetneqq X$ and $\sqrt{-1}R^{(L,h^L)}=\omega$ on $X\setminus K$. Then,
	\be
	H^j(X,K_X\otimes L)=0\quad \mbox{for all}~ j\geq q.
	\ee
\end{prop}
\begin{proof}
	 Let $\varphi\in \cC^\infty(X,\R)$ be an exhaustion function of $X$ such that $\sqrt{-1}\dbar\ddbar\varphi\geq 0$ on $X$ and $X_c:=\{ \varphi<c \}\Subset X$ for all $c\in \R$. We choose a regular value $c\in \R$ of $\varphi$ such that $K\subsetneqq X_c$ by Sard's theorem. Thus $X_c$ is a smooth pseudoconvex domain and $\sqrt{-1}R^L=\omega>0$ on a neighborhood of $bX_c$, in particular on $X_c\setminus K$. It follows that for $s\in \Omega^{n,j}(X_c,L)$, $j\geq q$,
	\be
	\langle[\sqrt{-1}R^L,\Lambda]s,s\rangle&=&\langle[\sqrt{-1}R^L,\Lambda]s,s\rangle_K+\langle[\omega,\Lambda]s,s\rangle_{X_c\setminus K}\geq \|s\|^2_{X_c\setminus K}.
	\ee
	If $s\in B^{n,j}(X_c, L)$,
	$\|\ddbar^Ls\|^2+\|\ddbar^{L*}s\|^2\geq \langle[\sqrt{-1}R^L,\Lambda]s,s\rangle+\int_{bM_c}\cL_\rho(s,s)dv_{X_c}$ by \cite{MM}.
	Since $X_c$ is pseudoconvex, $\int_{bM_c}\cL_\rho(s,s)dv_{X_c}\geq 0$. Since $\ddbar^{L*}_H=\ddbar^{L*}$ on $B^{0,j}(X_c,K_X\otimes L)$,
	\be
	\|\ddbar^Ls\|^2+\|\ddbar^{L*}_Hs\|^2\geq \|s\|^2-\|s\|^2_K
	\ee
	holds for $s\in B^{0,j}(X_c,K_X\otimes L)$, thus for $s\in \Dom(\ddbar^L)\cap\Dom(\ddbar^{L*})\cap L^2_{n,j}(X_c,L)$. In particular, if $s\in \cH^{n,q}(X_c,L)$, $s|_{X_c\setminus K}=0$ and so $\cH^{n,j}(X,L)=0$ for $j\geq q$. Since the fundamental estimate holds for $(0,j)$-form with values in $K_X\otimes L$ on $X_c$, $H^{0,j}_{(2)}(X_c,K_X\otimes L)=\cH^{0,j}(X_c,K_X\otimes L)=0$ for $j\geq q$. Moreover, by \cite[Theorem 1.2]{Ohs:15} and $\omega=\sqrt{-1}R^L$ on $X\setminus X_c$, it follows
	$H^j(X,K_X\otimes L)\cong H^{n,j}(X,L)\cong H^{n,j}_{(2)}(X_c,L)=0$.
\end{proof}

For a pseudoconvex domain $M$, we follow the above argument for $X_c$ and obtain:

\begin{prop}\label{P:pseudo_vanish}
	Let $M$ be a smooth pseudoconvex domain in a K\"{a}hler manifold $(X,\omega)$ of dimension $n$ and $1\leq q\leq n$. Let $(L,h^L)$ be a Nakano $q$-semipositive line bundle with respect to $\omega$ on $M$. Assume $(L,h^L)$ is Nakano $q$-positive with respect to $\omega$ on a neighborhood of $bM$. Then for every $j\geq q$,
	\be
	H^{0,j}_{(2)}(M,K_X\otimes L)=0.
	\ee
\end{prop}
  
\begin{proof}
	Let $(L,h^L)$ be Nakano $q$-positive with respect to $\omega$ on a neighbourhood $U$ of $bM$ such that $\ov U$ is compact. Let $V\Subset U$ and $V$ be a smaller neighborhood of $bM$.  By the Bochner-Kodaira-Nakano formula with boundary term \cite[Corollary 1.4.22]{MM}, for any $s\in B^{0,q}(M,L\otimes K_X)$, $k\geq 0$,
	\be\nonumber
	\frac{3}{2}\|\ddbar^{K_X} s\|^2+\|\ddbar^{K_X,*}s\|^2
	&\geq& \langle R^L(w_j,\ov w_k)\ov w^k\wedge i_{\ov w_j} s,s\rangle 
	\geq  \langle R^L(w_j,\ov w_k)\ov w^k\wedge i_{\ov w_j} s,s\rangle_{M\cap V}\\
	&\geq& C\|s\|^2_{M\cap V}=C(\|s\|^2-\|s\|^2_{M\setminus V}), 
	 \ee
	 where $C>0$, given by the Nakano $q$-positive line bundle $L$ with respect to $\omega$ on $U$ and the compactness of $\ov {M\cap V}\subset U$, is independent of the choice of $s$.
	Thus, we follow the argument for $X_c=M$ in Proposition \ref{prop_weak1_vanish} and obtain $\cH^{0,q}(M,L\otimes K_X)=0$. Since the fundamental estimate holds, $H^{0,q}_{(2)}(M,L\otimes K_X)=0$. And the assertion holds for all $j\geq q$ by Proposition \ref{prop_refine_semip} and Remark \ref{rem_j_pos}.
\end{proof}

\subsection{Remarks on $\omega$-trace and Kodaira type vanishing theorems}\label{sec_remark}
	
	Let $(E,h^E)$ be a holomorphic Hermitian vector bundle on a Hermitian manifold $(X,\omega)$. The $\omega$-trace of $R^{(E,h^E)}$ can be represented by 
	\be\nonumber
	\tau(E,h^E,\omega):=\mbox{Tr}_\omega R^{(E,h^E)}
	:=\sum_j R^{(E,h^E)}(\omega_j,\ov\omega_j)
	=\sum_{i,k} R^{(E,h^E)}\left(\frac{\dbar}{\dbar z_i},\frac{\dbar}{\dbar \ov z_k}\right)\langle dz_i,d\ov z_k\rangle_{g^{T^*X}}.  
	\ee  
Comparing to the usual trace $\mbox{Tr}[R^{(E,h^E)}]\in \Omega^{1,1}(X)$ depending only on $h^E$, $\tau(E,h^E,\omega):=\mbox{Tr}_\omega R^{(E,h^E)}\in \End(E)$ depends on $h^E$ and $\omega$. By Bochner-Kodaira-Nakano formulas, Serre duality and Le Potier's Theorem \cite[3.5.1, (3.5.8)]{Kob:87}, it follows that:
 
 \begin{prop}\label{thm_trace_vanish}
 	Let $(E,h^E)$ be a holomorphic Hermitian vector bundle over a compact K\"{a}hler manifold $(X,\omega)$.
 		(1) If $\tau(O_{E^*}(1))\leq 0$ and $<0$ at one point on $P(E^*)$, then
 	$	 H^{0}(X, S^m(E))=0$ for all $ m\geq 1$.	
 	(2) If $\tau(E)\leq 0$ and $<0$ at one point on $X$, then  
 	$	 H^{0}(X, E^m)=0$ for all $ m\geq 1$.
 \end{prop}   
 \begin{proof}
 	 The case $m=1$ of (2) follows from Bochner-Kodaira-Nakano formulas (or using the Lichnerowicz formular \cite[(1.4.31)]{MM})). From the fact $\tau(E^{\otimes m})=\tau(E)^{\otimes m}$, refer to \cite[III.(1.12)]{Kob:87} or \cite[(3.7)]{Yang:17}, we have $(2)$ holds for all $m\geq 1$. And $(1)$ is from Le Potier's Theorem \cite[3.5.1, (3.5.8)]{Kob:87} and $(2)$ for $E=O_{E^*}(1)$.
 \end{proof} 
 
Recall that a compact complex manifold X is said to be rationally connected if any two points of X can be joined by a chain of rational curves, see \cite{CDP:15}. We say a real $(1,1)$-form $\alpha\in \Omega^{1,1}(X)$ is quasi-positive on $X$, if $\alpha\geq 0$ on X and $>0$ at one point.
	\begin{prop}\label{thm_rc}
		Let $X$ be a compact K\"{a}hler manifold with a quasi-positive $(1,1)$-form representing the first Chern class $c_1(X)$. Then $X$ is projective and rationally connected. 
	\end{prop}
   
		\begin{proof}
			Calabi-Yau theorem \cite{Yau:78} provides a K\"{a}hler metric $\omega$ on $X$ such that the Ricci form $\sqrt{-1}R^{K_X^*}=\mbox{Ric}_\omega$ is quasi-positive, so $K_X^*$ is big and $X$ is projective.
				Since $(X,\omega)$ is K\"{a}hler, $\mbox{Ric}_\omega=\sqrt{-1}\mbox{Tr}[R^{T^{1,0}X}]$ and it coincides with
			$\tau(T^{1,0}X,h_\omega,\omega)=\mbox{Tr}_\omega R^{T^{1,0}X}$ as Hermitian matrices. Thus, 
				$\tau(T^{1,0}X,h_\omega,\omega)\geq 0$ and $>0$ at one point. By $\tau(T^{1,0}X)=-\tau(T^{1,0*}X)$ and Proposition \ref{thm_trace_vanish} (2), we have 
			$ H^{0}(X, (T^{1,0*}X)^m)=0~\mbox{for all}~ m\geq 1$, and the rationally connected follows from \cite[5.1. Corollary]{CDP:15}.
		\end{proof}
    
	Equivalently, it follows from \cite[Ch.III.\,(1.34)]{Kob:87} and \cite[5.1 Corollary]{CDP:15} that: A compact K\"{a}her manifold with quasi-positive Ricci curvature is projective and rationally-connected. It strengthens \cite[Theorem B (A)]{Wu:81} which asserted such a manifold is simply connected and has no nonzero holomorphic $q$-forms for $q>0$, since any rationally connected projective manifold has these properties, see \cite[Corollary 4.18]{Deb:01}. And it also leads to the fact \cite{Cam:92,KMM:92} that every smooth Fano manifold $X$ is rationally connected (See \cite{Deb:01,Ko:96,Yau:82problem}).

\begin{prop}\label{cor_sc}
	 Let $X$ be a compact K\"{a}hler manifold of nonnegative bisectional curvature. The following conditions are equivalent:
	 (A) $X$ is simply connected; (B) The first Betti number is zero;
	 (C) $X$ has quasi-positive Ricci curvature;
	 (D) $X$ is projective and rationally connected.
\end{prop}
\begin{proof}
	From \cite[Corollary 1]{HSW:81} and Proposition \ref{thm_rc}, we see (A), (B) and (C) are equivalent and (C) implies (D). And \cite[Corollary 4.18]{Deb:01}  entails (D) implies (A).
\end{proof}	
  
\section{Dirac operator on Nakano $q$-positive line bundles} \label{sec_spec} 

Inspired by \cite{M:92} and \cite[Theorem 1.1, 2.5]{MM:02}, we consider $q$-positive line bundles and the Dirac operators. We give some estimates of modified Dirac operators on high tensor powers of $q$-positive line bundles based on \cite[Section 1.5]{MM}. 
 
\subsection{Nakano $q$-positive line bundles with respect to $\omega$} 

 In this section, we work on the following setting. Let $(X,J)$ be a smooth manifold with almost complex structure $J$ and $\dim _{\R}X=2n$. Let $g^{TX}$ be a Riemannian metric compatible with $J$ and $\om:=g^{TX}(J\cdot,\cdot)$ be the real $(1,1)$-forms on $X$ induced by $g^{TX}$ and $J$. Let $(E,h^E)$ and $(L,h^L)$ be Hermitian vector bundles on $X$ with $\rank(L)=1$. Let $\nabla^E$ and $\nabla^L$ be Hermitian connections on $(E,h^E)$ and $(L,h^L)$ and let $R^E:=(\nabla^E)^2$ and $R^L:=(\nabla^L)^2$ be the curvatures. Assume that $\frac{\sqrt{-1}}{2\pi}R^L$ is compatible with $J$. Thus, the Chern-Weil form $c_1(L,h^L):=\frac{\sqrt{-1}}{2\pi}R^L$ representing the first Chern class $c_1(L)$ of $L$ is a real $(1,1)$-forms on $X$. (For example, $X$ is a compact complex manifold and $(E,h^E, \nabla^E),(L,h^L,\nabla^L)$ are holomoprhic Hermitian).
  
 The almost complex structure $J$ induced a splitting of the complexification of the tangent bundle, i.e., $TX\otimes \C=T^{1,0}X\bigoplus T^{0,1}X$, and the cotangent bundle. Let $0\leq p,q\leq n$, and let $\bigwedge^{p,q}T_x^*X$ be the fibre of the bundle $\bigwedge^{p,q}T^*X:=\wedge^pT^{1,0*}X\otimes\wedge^qT^{0,1*}X$ for $x\in X$.
 For $k\in \N$, we denote by $ \Omega^{p,q}(X,L^k\otimes E)$ the space of $(p,q)$-forms with values in $L^k\otimes E$ on $X$ and set $\Omega^{0,\geq q}(X,L^k\otimes E):=\bigoplus_{j\geq q}^n\Omega^{0,j}(X,L^k\otimes E)$. As defined in Sec. \ref{sec_prel}, we denote by $\langle\cdot,\cdot\rangle_h$ and $|\cdot|_h$ the pointwise Hermitian inner product and Hermitian norm, and by $\langle\cdot,\cdot\rangle$ and $\|\cdot\|$ the $L^2$ inner product and $L^2$-norm.  Let $\Lambda$ be the dual of the operator $\mL:=\omega\wedge\cdot$ on $\Omega^{p,q}(X)$ with respect to the Hermitian inner product $\langle\cdot,\cdot\rangle_h$ on $X$. In a local orthonormal frame $\{ w_j \}_{j=1}^n$ of $T^{1,0}X$ with respect to $g^{TX}$ and its dual $\{ w^j \}$ of $T^{1,0*}X$, $R^{(L,h^L)}=R^{(L,h^L)}(w_i,\ov w_j)w^i\wedge \ov w^j$, $\mL=\sqrt{-1}\sum_{j=1}^n w^j\wedge \ov w^j$ and $\Lambda=-\sqrt{-1}\sum_{j=1}^n i_{\ov w^j} i_{w^j}$. For any $s\in \Omega^{p,q}(X,L^k\otimes E)$,  $\left\langle[\sqrt{-1}R^{(L,h^L)}, \Lambda]s,s\right\rangle_h\in \cC^\infty(X,\R)$.
 We set 
 \be
 w_d=-\sum_{i,j}R^L(w_i,\ov w_j)\ov w^j\wedge i_{\ov w_i}\in \End(\Lambda (T^{*0,1}X)).
 \ee 
  
 For a real $3$-form $A$ on $X$, one can define modified Dirac operator $D^{c,A}_k$ acting on $\Omega^{0,\bullet}(X,L^k\otimes E)=\bigoplus_{j\geq 0}\Omega^{0,j}(X,L^k\otimes E)$, see \cite[Definition 1.3.6, (1.5.27)]{MM}.
  Ma-Marinescu obtained the precise lower bound of $D^{c,A}_k$ as follows. 
  The proof is based on a application of Lichnerowicz formula, see
  \cite[(1.5.34),(1.5.30)]{MM} and \cite{MM:02}. 
 \begin{thm}[\cite{MM}]\label{thm_mm}
 	There exists $C>0$ such that for any $k\in \N$, $s\in \Omega^{0,\bullet}(X,L^k\otimes E)$, 
 	\be
 	 	\|D^{c,A}_k s\|^2\geq 2k\langle-w_d s,s\rangle-C\|s\|^2.
 	\ee
 \end{thm} 
 As a consequence, they obtained the spectral gap property \cite[Theorem 1.5.7, 1.5.8]{MM}, which play the essential role in their approach to the Bergman kernel. In this section, we generalize \cite[Theorem 1.5.7]{MM} to the case of Nakano $q$-positive line bundles.
     
 \begin{defn}
 	For each $1\leq q\leq n$, the number $\mu_q\in \R\cup\{\pm\infty\}$ defined by
 	\be\label{eq_mu_def}
 	\mu_q(x):=\inf_{u\in \wedge^{n,q}T_x^*X}\frac{\langle [\sqrt{-1}R^L,\Lambda]u,u \rangle_h}{|u|^2_h},\quad
 	\mu_q:=\inf_{x\in X}\mu_q(x).
 	\ee
 \end{defn}
 
In terms of local orthonormal frame $\{\omega_j\}$ of $T^{1,0}X$, it follows that
\be \label{eq_loca_mu}
\mu_q
=\inf_{\alpha\in \wedge^{0,q}T_x^*X, x\in X}
\frac{\left\langle R^{L}(\omega_i,\ov\omega_j)\ov\omega^j\wedge i_{\ov \omega_i}\alpha,\alpha\right\rangle_h}{|\alpha|^2_h}
=\inf_{\alpha\in \wedge^{0,q}T_x^*X, x\in X}\frac{\langle -w_d \alpha,\alpha \rangle_h}{|\alpha|^2_h}.
\ee
In other words, if $\lambda_1(x)\leq \lambda_1(x)\leq \cdots\leq \lambda_n(x)$ are
the eigenvalues of $R^L_x$ with respect to $\omega$ at $x\in X$, 
then $\mu_q(x)=\sum_{j=1}^q\lambda_j(x)$ and $\mu_q=\inf_{x\in X}\mu_q(x)$.  
  \begin{thm}\label{thm_q_mod}
  	Let $X$ be compact. Let $1\leq q\leq n$ and $(L,h^L)$ is Nakano $q$-positive line bundle with respect to $\om$ on $X$. Then there exists $C_L>0$ such that for any $k\in \N$ and any $s\in \Omega^{0,\geq q}(X,L^k\otimes E)$,
  	\be
  	\|D_k^{c,A}s\|^2\geq (2\mu_qk-C_L)\|s\|^2,
  	\ee
  	where the constant $\mu_q>0$ defined in (\ref{eq_mu_def}).
  	Especially, for $k$ large enough,
  	\be
  	\Ker\left(D^{c,A}_k|_{\Omega^{0,\geq q}(X,L^k\otimes E)}\right)=0.
     \ee
  \end{thm}
  
  \begin{proof}
  	As in (\ref{eq_comp_starq}), we choose a local orthonormal frame around $x\in X$ such that 
  	$R_x^L(\omega_i,\omega_j)=\delta_{ij} c_i(x)$ for $1\leq i,j\leq n$. Then
  	\be\label{eq_w_d}
  	w_d=-\sum_{j\geq 1} c_j(x)\ov w^j\wedge i_{\ov w_j}\in \End(\Lambda (T_x^{*0,1}X)).
  	\ee
  	 Let $C_J(x):=\sum_{j\in J}c_j(x)$ for each ordered $J=(j_1,\cdots,j_q)$ with $|J|= q$.
  	 For $\alpha\in \bigwedge^{0,q}_xX\setminus\{ 0\}$, we represent it by $\alpha=\sum _J\alpha_{J} \ov w^J$ with $|J|=q$. From (\ref{eq_loca_mu}) and (\ref{eq_w_d}), we have
  	 \be\label{eq_mu_CJ}
  	 \mu_q=\inf_{x\in X}\inf_{\alpha_J\in \C}\frac{\sum_J C_J(x)|\alpha_J|^2}{\sum_J|\alpha_J|^2} =\inf_{x\in X}\inf_{|J|=q}C_J(x).
  	 \ee
  	 
  	 For $s(x)\in \bigwedge^{0,q}_xX\otimes L_x^k\otimes E_x$, we can represent it by $s(x)=\sum _{J,i}s_{J,i}(x) \ov w^J\otimes e_i^k$ for a local orthonormal frame $\{ e^k_i \}$ of $L^k\otimes E$. Thus $|s(x)|^2_h=\sum_{J,i}|s_{J,i}(x)|^2$. By (\ref{eq_w_d}) and (\ref{eq_mu_CJ}), Theorem \ref{thm_mm} entails that, for any $s\in \Omega^{0,q}(X,L^k\otimes E)$,
  	 \be
  	 	\|D^{c,A}_k s\|^2&\geq& 2k\langle-w_d s,s\rangle-C\|s\|^2
  	 	=2k\int_X\sum_{J,i}C_J(x)|s_{J,i}|^2 dv_X-C\|s\|^2
  	 \ee
  	 
  	By (\ref{hypo_semi}), (\ref{D:q-semi-local}) and (\ref{eq_mu_def}), we have $\mu_q> 0$. 
  	By (\ref{eq_mu_CJ}), it follows that
  	\be
  	\|D^{c,A}_k s\|^2\geq 2k\mu_q\|s\|^2-C\|s\|^2
  	\ee
  	holds for $s\in \Omega^{0, q}(X,L^k\otimes E)$.
  	By Proposition \ref{prop_refine_semip} and Remark \ref{rem_j_pos}, we see $\mu_{j+1}> \mu_j>0$ for each $j\geq q$. Thus the assertion holds for $s\in \Omega^{0,\geq q}(X,L^k\otimes E)$. 
  \end{proof}
 
 \begin{rem}
 	From Remark \ref{rem_j_pos}, the positive assumption \cite[(1.5.21)]{MM} is equivalent to Nakano $1$-positive line bundle with respect to $\omega$. By (\ref{eq_mu_CJ}), $\mu_1=\inf_{x\in X,1\leq j\leq n}c_j(x)$. Thus
 	\cite[Theorem 1.5.7]{MM} follows from Theorem \ref{thm_q_mod} by choosing $q=1$. 
 \end{rem}
   
In general, for a real $3$-form $A$ on $X$, $(D_k^{c,A})^2$ may not preserve the
$\Z$-grading of $\Omega^{0,\bullet}(X,L^k\otimes E)$. As a special case, we can consider Kodaira Laplacian $\square^{L^k\otimes E}$, which preserves the $\Z$-grading. 
  	On a complex manifold $X$, Hodge-Dolbeault operator 
	satisfies $D_k:=\sqrt{2}(\ddbar^{L^k\otimes E}+\ddbar^{L^k\otimes E,*})=D^{c,A}_k$, 
	for $A=-\frac{1}{4}T_{as}$, see \cite[(1.4.17)]{MM}, 
	and the Kodaira Laplacian satisfies $\square^{L^k\otimes E}=\frac{1}{2}D_k^2$. 
   Then, from Hodge theorem, Serre duality and the equivalent definition of the $q$-positive line bundle (see Remark \ref{D:q-pos-M}), Theorem \ref{thm_q_mod} 
   leads to Andreotti-Grauert vanishing theorem \cite[Proposition 27]{AG:62} 
   (see also \cite[(5.1) Theorem]{Dem}):
  \begin{cor}[\cite{AG:62}]\label{cor_ag}
  	Let $X$ be a compact complex manifold of dimension $n$ and $(E,h^E)$ and $(L,h^L)$ be holomorphic Hermitian vector bundles on $X$ with $\rank(L)=1$. If $R^L$ has at least $p$ positive eigenvalues and at least $q$ negative eigenvalues at every $x\in X$, then, for $j\in \{ j\in \N: j\leq q-1 ~\mbox{or}~ j\geq n-p+1 \}$ and sufficiently large $k$, 
  $	H^j(X,L^k\otimes E)=0$. 
  \end{cor} 	
        	
By the same argument in \cite[Theorem 4.4,Corollary 4.5-4.6]{MM:02} and \cite[(6.1.15)]{MM}, Theorem \ref{thm_q_mod} still holds on $\Gamma$-covering manifolds as follows. Let $\til X$ be a $\Gamma$-covering manifold of dimension $n$. Let  $\til J$ be $\Gamma$-invariant almost complex structure on $\til X$. Let $g^{T\til X}$ be a $\Gamma$-invariant Riemannian metric compatible with $\til J$ and $\om:=g^{T\til X}(\til J\cdot,\cdot)$ be the real $(1,1)$-forms on $\til X$ induced by $g^{T\til X}$ and $\til J$. Let $(\til E,h^{\til E})$ and $(\til L,h^{\til L})$ be $\Gamma$-invariant holomorphic Hermitian vector bundles on $\til X$ with $\rank(\til L)=1$. Let $\nabla^{\til E}$ and $\nabla^{\til L}$ be Chern connections on $(\til E,h^{\til E})$ and $(\til L,h^{\til L})$ and let $R^{\til E}:=(\nabla^{\til E})^2$ and $R^{\til L}:=(\nabla^{\til L})^2$ be the curvatures. 
Let $\til D_k:=\sqrt{2}(\ddbar^{\til L^k\otimes \til E}+\ddbar^{\til L^k\otimes \til E,*})$ be the Hodge-Dolbeault operator defined on $\Dom(\til D_k)=\Dom(\ddbar^{\til L^k\otimes \til E})\cap \Dom(\ddbar^{\til L^k\otimes \til E, *})$ and $\square^{\til L^k\otimes \til E}:=\frac{1}{2}\til D_k^2$ the self-adjoint extension of Kodaira Laplacian.
\begin{thm}\label{thm_main_spectral_cover}
	Assume $1\leq q\leq n$ and $(\til L,h^{\til L})$ is Nakano $q$-positive with respect to $\om$ on $\til X$.  Then there exists $C_{\til L}>0$ such that for  any $k\in \N$ and any $\til s\in \Dom(\til D_k)\cap L^2_{0,\geq q}(\til X,\til{L}^k\otimes \til E)$, 
	\be 
	\|\til D_k\til s\|^2\geq (2\mu_qk-C_{\til L})\|\til s\|^2,
	\ee 
	where the constant $\mu_q>0$ defined in (\ref{eq_mu_def}).
\end{thm}
For the $L^2$ Andreotti-Grauert theorem on covering manifolds, see \cite[Theorem 3.5]{Bra:99} and \cite[Section 4]{MM:02}. 
\subsection{Semipositive line bundles of type $q$}
Let $X$ be a complex manifold of dimension $n$ and $(L,h^L)$ be a holomorphic Hermitian line bundle. 
For $1\leq q\leq n$, we have the notion of semipositive line bundles of type $q$ as follows,  refer to \cite[Chapter 3, Section 1, Definition 1.1]{Oh:82}. 
	We say $(L,h^L)$ is semipositive of type $q$ if $(L,h^L)\geq 0$ everywhere and $\sqrt{-1}R_x^{(L,h^L)}$ is positive on a $(n-q+1)$-dimensional subspace of $T_x^{(1,0)}X$ at every $x\in X$,

 We remark that, by \cite[Chapter 3, Section 2, Proposition 2.1 (1),(2)] {Oh:82} and Definition \ref{def_qsemip}, if $(L,h^L)$ is semipositive of type $q$ on a complex manifold $X$, then $(L,h^L)$ is Nakano $q$-positive at every point $x\in X$ with respect to arbitrary Hermitian metric $\omega$ on $X$. As a consequence, by replacing the hypothesis Nakano $q$-positive with respect to $\omega$ by semipositive of type $q$ 
 in Theorem \ref{thm_q_mod} and \ref{thm_main_spectral_cover}, the conclusion therein still hold.
 
Besides, by adapting the notion of semipositive of type $q$ to Theorem \ref{T:qvsh}, we obtain another generalization of \cite[Theorem 3.5.9]{MM} as follows. 
\begin{cor}
	Let $(X,\omega)$ be a $q$-convex manifold of dimension $n$. Let $E, L$ be holomorphic vector bundle with $\rank(L)=1$. 
	Let $K\subset X$ be the exceptional set and $1\leq p\leq n$. If $(L,h^L)$ is semipositive of type $p$ on $X_c$ with $K\subset X_c$, then for $j\geq \max\{p,q\}$ and $k$ sufficiently large,
	\begin{equation}  
		H^j(X,L^k\otimes E)=0.
	\end{equation}
\end{cor}	 
\begin{proof} 
	Since $(L,h^L)$ is semipositive of type $p$  on $X_c$, $(L,h^L)$ is Nakano $q$-positive with respect to $\omega$ given by Lemma \ref{lowbd_rho_lem} on $X_c$. Finally, we use Theorem \ref{T:qvsh}.
\end{proof}	

\begin{cor}
	Let $M$ be a smooth pseudoconvex domain in a K\"{a}hler manifold $(X,\omega)$ of dimension $n$ and $1\leq q\leq n$. Let $(L,h^L)$ be a semipositive line bundle on $M$. Assume $(L,h^L)$ is semipositive of type $q$ on a neighborhood of $bM$. Then for every $j\geq q$,
	\be
	H^{0,j}_{(2)}(M,K_X\otimes L)=0.
	\ee
\end{cor} 
\begin{proof}
Proposition \ref{P:semipos} and Proposition \ref{P:pseudo_vanish} and the fact that $(L,h^L)$ is Nakano $q$-positive with respect to any Hermitian metric $\omega$ on a neighborhood of $bM$.
\end{proof} 
 
\section*{Acknowledgements} 
I thank Prof.\ Chin-Yu Hsiao and Prof.\ Rung-Tzung Huang for
support and encouragement. This work is partially supported by MOST 108-2811-M-008-500. 

\begin{thebibliography}{BDPP13}
	
	\bibitem[AG62]{AG:62}
	Aldo Andreotti and Hans Grauert.
	\newblock Th\'eor\`eme de finitude pour la cohomologie des espaces complexes.
	\newblock {\em Bull. Soc. Math. France}, 90:193--259, 1962.
	
	\bibitem[Ati76]{Atiyah:76}
	Michael~F. Atiyah.
	\newblock Elliptic operators, discrete groups and von {N}eumann algebras.
	\newblock pages 43--72. Ast\'erisque, No. 32--33, 1976.
	
	\bibitem[AV65]{AV:65}
	Aldo Andreotti and Edoardo Vesentini.
	\newblock Carleman estimates for the {L}aplace-{B}eltrami equation on complex
	manifolds.
	\newblock {\em Inst. Hautes \'Etudes Sci. Publ. Math.}, (25):81--130, 1965;
	Erratum: (27):153-155,1965.
	
	\bibitem[BDPP13]{BBDP:13}
	S\'ebastien Boucksom, Jean-Pierre Demailly, Mihai P\u{a}un, and Thomas
	Peternell.
	\newblock The pseudo-effective cone of a compact {K}\"ahler manifold and
	varieties of negative {K}odaira dimension.
	\newblock {\em J. Algebraic Geom.}, 22(2):201--248, 2013.
	
	\bibitem[Ber02]{BB:02}
	Bo~Berndtsson.
	\newblock An eigenvalue estimate for the {$\overline\partial$}-{L}aplacian.
	\newblock {\em J. Differential Geom.}, 60(2):295--313, 2002.
	
	\bibitem[Bou89]{Bouche:89}
	Thierry Bouche.
	\newblock In\'{e}galit\'{e}s de {M}orse pour la {$d''$}-cohomologie sur une
	vari\'{e}t\'{e} holomorphe non compacte.
	\newblock {\em Ann. Sci. \'{E}cole Norm. Sup. (4)}, 22(4):501--513, 1989.
	
	\bibitem[Bra99]{Bra:99}
	Maxim Braverman.
	\newblock Vanishing theorems on covering manifolds.
	\newblock In {\em Tel {A}viv {T}opology {C}onference: {R}othenberg
		{F}estschrift (1998)}, volume 231 of {\em Contemp. Math.}, pages 1--23. Amer.
	Math. Soc., Providence, RI, 1999.
	
	\bibitem[Cam92]{Cam:92}
	Fr{\'e}d{\'e}ric Campana.
	\newblock Connexit\'{e} rationnelle des vari\'{e}t\'{e}s de {F}ano.
	\newblock {\em Ann. Sci. \'{E}cole Norm. Sup. (4)}, 25(5):539--545, 1992.
	
	\bibitem[CD01]{CD:01}
	Fr{\'e}d{\'e}ric Campana and Jean-Pierre Demailly.
	\newblock Cohomologie {$L^2$} sur les rev\^etements d'une vari\'et\'e complexe
	compacte.
	\newblock {\em Ark. Mat.}, 39(2):263--282, 2001.
	
	\bibitem[CDP15]{CDP:15}
	Fr{\'e}d{\'e}ric Campana, Jean-Pierre Demailly, and Thomas Peternell.
	\newblock Rationally connected manifolds and semipositivity of the {R}icci
	curvature.
	\newblock In {\em Recent advances in algebraic geometry}, volume 417 of {\em
		London Math. Soc. Lecture Note Ser.}, pages 71--91. Cambridge Univ. Press,
	Cambridge, 2015.
	
	\bibitem[CM15]{CM}
	Dan Coman and George Marinescu.
	\newblock Equidistribution results for singular metrics on line bundles.
	\newblock {\em Ann. Sci. \'Ec. Norm. Sup\'er. (4)}, 48(3):497--536, 2015.
	
	\bibitem[Deb01]{Deb:01}
	Olivier Debarre.
	\newblock {\em Higher-dimensional algebraic geometry}.
	\newblock Universitext. Springer-Verlag, New York, 2001.
	
	\bibitem[Dem]{Dem}
	Jean-Pierre Demailly.
	\newblock Complex analytic and differential geometry. 
	\newblock {\em Available at
		http://www-fourier.ujf-grenoble.fr/~demailly/manuscripts/agbook.pdf}.
	
	\bibitem[Dem85]{Dem:85}
	Jean-Pierre Demailly.
	\newblock Champs magn\'etiques et in\'egalit\'es de {M}orse pour la
	{$d''$}-cohomologie.
	\newblock {\em Ann. Inst. Fourier (Grenoble)}, 35(4):189--229, 1985.
	
	\bibitem[Dem86]{Dem:86}
	Jean-Pierre Demailly.
	\newblock Sur l'identit\'e de {B}ochner-{K}odaira-{N}akano en g\'eom\'etrie
	hermitienne.
	\newblock In {\em S\'eminaire d\/'analyse {P}. {L}elong-{P}. {D}olbeault-{H}.
		{S}koda, ann\'ees 1983/1984}, volume 1198 of {\em Lecture Notes in Math.},
	pages 88--97. Springer, Berlin, 1986.
	
	\bibitem[Dem12]{Dem:12}
	Jean-Pierre Demailly.
	\newblock {\em Analytic methods in algebraic geometry}, volume~1 of {\em
		Surveys of Modern Mathematics}.
	\newblock International Press, Somerville, MA; Higher Education Press, Beijing,
	2012.
	
	\bibitem[Fuj12]{Fujino:12}
	Osamu Fujino.
	\newblock A transcendental approach to {K}oll\'{a}r's injectivity theorem.
	\newblock {\em Osaka J. Math.}, 49(3):833--852, 2012.
	
	\bibitem[Gri66]{Griffs:66}
	Phillip~A. Griffiths.
	\newblock The extension problem in complex analysis. {II}. {E}mbeddings with
	positive normal bundle.
	\newblock {\em Amer. J. Math.}, 88:366--446, 1966.
	
	\bibitem[HL88]{HL:88}
	Gennadi~M. Henkin and J{\"u}rgen Leiterer.
	\newblock {\em Andreotti-{G}rauert theory by integral formulas}, volume~74 of
	{\em Progress in Mathematics}.
	\newblock Birkh\"auser Boston, Inc., Boston, MA, 1988.
	
	\bibitem[HSW81]{HSW:81}
	Alan Howard, Brian Smyth, and Hung-Hsi Wu.
	\newblock On compact {K}\"{a}hler manifolds of nonnegative bisectional
	curvature. {I}.
	\newblock {\em Acta Math.}, 147(1-2):51--56, 1981.
	
	\bibitem[KMM92]{KMM:92}
	J\'{a}nos Koll\'{a}r, Yoichi Miyaoka, and Shigefumi Mori.
	\newblock Rational connectedness and boundedness of {F}ano manifolds.
	\newblock {\em J. Differential Geom.}, 36(3):765--779, 1992.
	
	\bibitem[Kob87]{Kob:87}
	Shoshichi Kobayashi.
	\newblock {\em Differential geometry of complex vector bundles}, volume~15 of
	{\em Publications of the Mathematical Society of Japan}.
	\newblock Princeton University Press, Princeton, NJ; Iwanami Shoten, Tokyo,
	1987.
	\newblock Kan{\^o} Memorial Lectures, 5.
	
	\bibitem[Kod54]{Kod:54}
	Kunihiko Kodaira.
	\newblock On {K}\"{a}hler varieties of restricted type (an intrinsic
	characterization of algebraic varieties).
	\newblock {\em Ann. of Math. (2)}, 60:28--48, 1954.
	
	\bibitem[Kol96]{Ko:96}
	J\'anos Koll\'ar.
	\newblock {\em Rational curves on algebraic varieties}, volume~32 of {\em
		Ergebnisse der Mathematik und ihrer Grenzgebiete. 3. Folge. A Series of
		Modern Surveys in Mathematics [Results in Mathematics and Related Areas. 3rd
		Series. A Series of Modern Surveys in Mathematics]}.
	\newblock Springer-Verlag, Berlin, 1996.
	   
		\bibitem[MM02]{MM:02}
		Xiaonan Ma and George Marinescu.
		\newblock The {${\rm Spin}^c$} {D}irac operator on high tensor powers of a line
		bundle.
		\newblock {\em Math. Z.}, 240(3):651--664, 2002.
		
		\bibitem[MM07]{MM}
		Xiaonan Ma and George Marinescu.
		\newblock {\em Holomorphic {M}orse inequalities and {B}ergman kernels}, volume
		254 of {\em Progress in Mathematics}.
		\newblock Birkh\"auser Verlag, Basel, 2007.
		
	\bibitem[Mar92]{M:92}
	George Marinescu.
	\newblock Morse inequalities for {$q$}-positive line bundles over weakly
	{$1$}-complete manifolds.
	\newblock {\em C. R. Acad. Sci. Paris S\'er. I Math.}, 315(8):895--899, 1992.
	
	\bibitem[Mat14]{Mats:14}
	Shin-ichi Matsumura.
	\newblock A {N}adel vanishing theorem via injectivity theorems.
	\newblock {\em Math. Ann.}, 359(3-4):785--802, 2014.
	
	\bibitem[Mat16]{Mats:16}
	Shin-ichi Matsumura.
	\newblock A vanishing theorem of {K}oll\'ar-{O}hsawa type.
	\newblock {\em Math. Ann.}, 366(3-4):1451--1465, 2016.
	
	\bibitem[Mat18]{Mat:18JAG}
	Shin-ichi Matsumura.
	\newblock An injectivity theorem with multiplier ideal sheaves of singular
	metrics with transcendental singularities.
	\newblock {\em J. Algebraic Geom.}, 27(2):305--337, 2018.
	
	\bibitem[MTC02]{MTC:02}
	George Marinescu, Radu Todor, and Ionu{\c{t}} Chiose.
	\newblock {$L^2$} holomorphic sections of bundles over weakly pseudoconvex
	coverings.
	\newblock {\em Geom. Dedicata}, 91:23--43, 2002.
	
	\bibitem[Ohs82]{Oh:82}
	Takeo Ohsawa.
	\newblock Isomorphism theorems for cohomology groups of weakly {$1$}-complete
	manifolds.
	\newblock {\em Publ. Res. Inst. Math. Sci.}, 18(1):191--232, 1982.
	
	\bibitem[Ohs05]{Ohs:05}
	Takeo Ohsawa.
	\newblock On a curvature condition that implies a cohomology injectivity
	theorem of {K}oll\'{a}r-{S}koda type.
	\newblock {\em Publ. Res. Inst. Math. Sci.}, 41(3):565--577, 2005.
	
	\bibitem[Ohs15]{Ohs:15}
	Takeo Ohsawa.
	\newblock A remark on {H}\"{o}rmander's isomorphism.
	\newblock In {\em Complex analysis and geometry}, volume 144 of {\em Springer
		Proc. Math. Stat.}, pages 273--280. Springer, Tokyo, 2015.
	
	\bibitem[Siu82]{Siu:82}
	Yum~Tong Siu.
	\newblock Complex-analyticity of harmonic maps, vanishing and {L}efschetz
	theorems.
	\newblock {\em J. Differential Geom.}, 17(1):55--138, 1982.
	
	\bibitem[Siu84]{Sil:84}
	Yum~Tong Siu.
	\newblock A vanishing theorem for semipositive line bundles over non-{K}\"ahler
	manifolds.
	\newblock {\em J. Differential Geom.}, 19(2):431--452, 1984.
	
	\bibitem[Siu85]{Siu:85}
	Yum~Tong Siu.
	\newblock Some recent results in complex manifold theory related to vanishing
	theorems for the semipositive case.
	\newblock In {\em Workshop {B}onn 1984 ({B}onn, 1984)}, volume 1111 of {\em
		Lecture Notes in Math.}, pages 169--192. Springer, Berlin, 1985.
	
	\bibitem[TCM01]{TCM:01}
	Radu Todor, Ionu{\c{t}} Chiose, and George Marinescu.
	\newblock Morse inequalities for covering manifolds.
	\newblock {\em Nagoya Math. J.}, 163:145--165, 2001.
	
	\bibitem[Wan16]{Wh:16}
	Huan Wang.
	\newblock On the growth of von {N}eumann dimension of harmonic spaces of
	semipositive line bundles over covering manifolds.
	\newblock {\em Internat. J. Math.}, 27(11):1650093, 14, 2016.
	 
	\bibitem[Wan18]{Wh:17}
	Huan Wang.
	\newblock The growth of dimension of cohomology of semipositive line bundles on
	hermitian manifolds.
	\newblock {\em arXiv:1810.09881v1}, Oct. 2018.
	  
	\bibitem[Wu81]{Wu:81}
	Hung-Hsi Wu.
	\newblock On compact {K}\"{a}hler manifolds of nonnegative bisectional
	curvature. {II}.
	\newblock {\em Acta Math.}, 147(1-2):57--70, 1981.
	
	\bibitem[WZ19]{WZ:19}
	Zhi-Wei Wang and Xiang-Yu Zhou.
	\newblock Asymptotic estimate of cohomology groups valued in pseudo-effective
	line bundles.
	\newblock {\em arXiv:1905.03473v1}, May 2019.
	
	\bibitem[Yan18]{Yang:17}
	Xiaokui Yang.
	\newblock R{C}-positivity, rational connectedness and {Y}au's conjecture.
	\newblock {\em Camb. J. Math.}, 6(2):183--212, 2018.
	
	\bibitem[Yau78]{Yau:78}
	Shing~Tung Yau.
	\newblock On the {R}icci curvature of a compact {K}\"{a}hler manifold and the
	complex {M}onge-{A}mp\`ere equation. {I}.
	\newblock {\em Comm. Pure Appl. Math.}, 31(3):339--411, 1978.
	
	\bibitem[Yau82]{Yau:82problem}
	Shing~Tung Yau.
	\newblock Problem section.
	\newblock In {\em Seminar on {D}ifferential {G}eometry}, volume 102 of {\em
		Ann. of Math. Stud.}, pages 669--706. Princeton Univ. Press, Princeton, N.J.,
	1982.
	
\end{thebibliography}

\end{document}